\documentclass[final,3p,11pt,times]{elsarticle}
\usepackage{amssymb}
 \usepackage{amsthm}

\usepackage{graphicx} 
\usepackage{epstopdf}
\usepackage{multirow}
\usepackage{mathtools}
\usepackage{subfigure}

\usepackage{fullpage}
\usepackage[colorlinks]{hyperref}

\allowdisplaybreaks

\newtheorem{remark}{Remark}
\newtheorem{lemma}{Lemma}
\newtheorem{theorem}{Theorem}

\journal{}

\begin{document}

\begin{frontmatter}



\title{Positive-preserving, mass conservative linear schemes for the Possion-Nernst-Planck equations}

	\author[1]{Jiayin Li}
	\ead{ljy@pku.edu.cn}
	
	\author[1]{Jingwei Li}
	\ead{lijingwei@lzu.edu.cn}
%
%
%
\address[1]{School of Mathematics and Statistics, Lanzhou University, Lanzhou 730000, Gansu, China}


\begin{abstract}
The first-order linear positivity preserving schemes in time are available for the time dependent Poisson-Nernst-Planck (PNP) equations, second-order linear ones are still challenging. In this paper, we propose the first- and second-order exponential time differencing schemes with the finite difference spatial discretization for PNP equations, based on the Slotboom transformation of the Nernst–Planck equations. The proposed schemes are linear and preserve the mass conservation and positivity preservation of ion concentration at full
discrete level without any constraints on the time step size. The corresponding energy stability analysis is also presented, demonstrating that the second-order scheme can dissipate the modified energy. Extensive numerical results are carried out to support the theoretical findings and showcase the performance of the proposed schemes.
\end{abstract}



\begin{keyword}
PNP equations \sep Slotboom transformation \sep Structure-preserving  \sep Quasi-symmetric finite difference 

\MSC[2020] 35B50 \sep 35K55 \sep 65M12 \sep 65R20

\end{keyword}

\end{frontmatter}


\section{Introduction}
In this paper, we consider a time-dependent system of Possion-Nernst-Planck (PNP) equations in the dimensionless form \cite{PS09}
\begin{subequations}\label{PNP}
\begin{align}
p_t =& \nabla\cdot(\nabla p+p\nabla\phi), \, \mathbf{x}\in\Omega, t\in[0,T],\label{PNP_p}\\
n_t =& \nabla\cdot(\nabla n-n\nabla\phi), \,~ \mathbf{x}\in\Omega, t\in[0,T],\label{PNP_n}\\
-\epsilon^2\Delta \phi =& p-n+\rho^f, \,~~~~~~~~ \mathbf{x}\in\Omega, t\in[0,T],\label{PNP_phi}
\end{align}
\end{subequations}
subject to the  initial conditions
\begin{equation*}
\begin{aligned}
p(\mathbf{x},0)=p_{0}(\mathbf{x}),  \, n(\mathbf{x},0)=n_{0}(\mathbf{x}),\quad\mathbf{x}\in\Omega,
\end{aligned}
\end{equation*}
and either periodic boundary condition or homogeneous Neumann boundary conditions, where $p$ and $n$ are the concentration of positive and negative ions with valence +1 and -1 satisfying the positivity preservation, i.e. $p,n\geq 0$, and $\phi$ is the electric potential and $\rho^f$ is the density of fixed charges. $\Omega\in \mathbb{R}^{\rm d}$ is a two-dimensional rectangular domain (${\rm d}=2$) or three-dimensional cube domain (${\rm d}=3$) with Lipschitz continuous boundary $\partial\Omega$ and $T>0$ is the finite time. $\epsilon$ is a small positive dimensionless number related to the ratio of the Debye length to the physical characteristic length. 

Imposed by the periodic boundary condition, the PNP equations \eqref{PNP} satisfy the mass conservation for both ion species, i.e. 
\begin{align*}
\int_\Omega p {\rm d}\mathbf{x}=\int_\Omega p_0 {\rm d}\mathbf{x}, \int_\Omega n {\rm d}\mathbf{x}=\int_\Omega n_0 {\rm d}\mathbf{x}.
\end{align*}
Then we observe from \eqref{PNP_phi} that  
\begin{align*}
\int_{\Omega} p-n+\rho^f {\rm d}\mathbf{x}=0,
\end{align*}
which indicates that $p-n+\rho^f$ is of mean zero. In this case, the PNP equations \eqref{PNP} can be acted as the $H^{-1}$ gradient flow with respect to the energy functional
\begin{align}\label{egy}
E(t)=\int_{\Omega} p\ln p+n\ln n +\frac{\epsilon^2}{2}|\nabla\phi|^2 {\rm d}\mathbf{x} 
\end{align}
under the assumption that $p-n+\rho^f$ is of mean zero. Thus PNP equation can be rewritten as 
\begin{subequations}
\begin{align}
p_t =& \nabla\cdot(p\nabla \frac{\delta E}{\delta p}), \\
n_t =& \nabla\cdot(n\nabla \frac{\delta E}{\delta n}), \\
\phi =& (-\epsilon^2\Delta)^{-1}(p-n+\rho^f).
\end{align}
\end{subequations}
PNP equation can preserve the energy dissipation law $\frac{d}{dt}E(t) \leq 0$. Moreover, the solutions to the PNP equations preserve the positivity of density of positive and negative ions, i.e., 
\begin{align*}
p_0(\mathbf{x}),n_0(\mathbf{x})\geq 0, \forall\,\mathbf{x}\in\Omega\rightarrow p(\mathbf{x},t),n(\mathbf{x},t)\geq 0, \forall\,(\mathbf{x},t)\in\Omega\times (0,T].
\end{align*}

The PNP equations are widely used to model the diffusive behavior of charge particles under the effect of electric field arising from the various fields such as biological membrane channels \cite{CE93,E98,T53}, electrochemical systems \cite{BTA04}, and semiconductor devices \cite{FRB83,MRS12}. There are much efforts devoted to the construction on the numerical methods to simulate the PNP equations. For the spatial discretization, a partial list includes finite element method, finite difference method and finite volume method, see \cite{LWWY21,MG14,ST22,XC20,XL20} and references therein. For the time integration, the commonly used structure-preserving numerical schemes can be split into two categories. One is based on the observation that PNP equations can be viewed as a Wasserstein gradient flow \cite{KMX17,LM23,MXL16} of the free energy \eqref{egy}. Thus implicit-explicit treatment based on the convex-concave decomposition of the free energy naturally shares the energy dissipation \cite{LWWY23,QQSZ24,SX21}. The rigorous proof on the positivity of ionic concentration leverages the singular nature of logarithmic term, which preserves the numerical solutions reaching a singular point. Such a technique has been successfully applied to the Cahn-Hilliard equation with logarithmic potential \cite{CWWW19}. The other category is based on the Slotboom transformation of the Nernst-Planck equation into a self-adjoint elliptic operator \cite{BCV14}. The advantage of such a reformulation is to allow the quasi-symmetric discretization for the self-adjoint elliptic operator, which greatly facilitates the design of discretization schemes which are able to preserve the discrete maximum principle \cite{KHTC21}. Some first-order linear numerical schemes have been developed to unconditionally satisfy the positivity preservation and mass conservation \cite{DWZ23,DHQZ24,HPY19,HH20,LM21,LW14}, but the linear/nonlinear second-order Crank-Nicolson scheme only preserve the positivity conditionally \cite{DWZ19,DZ24}. Recently, a functional transformation method incorporating the scalar auxiliary variable approach was studied in \cite{HS21} to preserve the positivity of ionic concentration. A discontinuous Garlekin method was developed in \cite{LWYY22} for PNP equations in which the positivity of ionic concentration is realized by a positivity-preserving limiter. A new Lagrange multiplier method has been introduced in \cite{CS22a,CS22b} to construct the bound/positivity preserving schemes for parabolic equations and then a projection post-processing strategy is proposed in \cite{TC24} to ensure the physical constraint of PNP equations including positivity preservation and mass conservation. Though several efforts have focused on the construction of first-order unconditional positivity/bound preserving schemes, there remains a significant gap in the high order numerical schemes with unconditional preservation of physical properties. Thus it is still desirable to design the linear, high-order and structure preserving schemes for solving the PNP equations.

Exponential time differencing (ETD) method has emerged as an efficient approach for temporal integration, particularly in preserving the maximum bound principle for the Allen-Cahn equation at the discrete level in combination with linear stabilization techniques. The pioneering work in \cite{DJLQ19} proposed the first- and second-order stabilized ETD schemes to preserve the maximum bound principle unconditionally for the nonlocal Allen–Cahn equation, and then an abstract framework on MBP-preserving ETD schemes was established in \cite{DJLQ21} for a class of semilinear parabolic equations. Since then, such a theoretical framework has been successfully applied to the conservative Allen-Cahn equation \cite{JJLL21,LJCF21} and convective Allen-Cahn equation \cite{CJLL23,LLCJ23,LCLJ23}. In addition, an arbitrarily high-order ETD multi-step method was explored in \cite{LYZ20,YYZ21} incorporating a cut-off post-processing strategy to preserve the maximum bound principle while maintaining the numerical accuracy. 

Inspired by the unconditional maximum bound principle of ETD schemes and discrete maximum principle of the Slootboom transformation for the Nernst-Planck equation, we develop the first- and second-order ETD schemes for PNP equations, where the spatial discretization is adopted by the finite difference method. Though many existing works mainly focus on the first-order unconditional positivity preserving schemes, we make an exploration on the high-order schemes for the positivity preservation of PNP equations. Our proposed schemes are linear and proven to be mass conservative and positivity preserving in the discrete level without any constraints on the time step size. The corresponding energy stability analysis is also presented, demonstrating that the second-order scheme can dissipate the modified energy. To the best of our knowledge, it will be a first work on the linear second-order temporal accurate scheme to preserve three physical properties unconditionally for the PNP equations. 

The rest of this paper is structured as follows. Section \ref{sec:spatial} introduces the spatial semi-discrete system of PNP equations based on the Slotboom transformation. In Section \ref{sec:temporal}, we develop the first- and second-order ETD schemes and then give the rigorous proof on the unconditional positivity preservation and mass conservation in the discrete level. The corresponding energy stability analysis is established in Section \ref{sec:egy}. Several numerical examples are carried out to verify the theoretical results in Section \ref{sec:num}. Finally, we end this paper with some concluding remarks in Section \ref{sec:con}.
\section{Preliminaries}\label{sec:spatial}
In this section, we first give a review on the finite difference discretization, and then obtain the stabilized form of spatial-semidiscrete system.
\subsection{Spatial discretization} 
Assume the domain $\Omega=(0,L_x)\times(0,L_y)$, where for simplicity, we assume $L_x=L_y=L$. Let $h=L/N$ $(N\in \mathbb{Z}^+)$ be the mesh size, the uniform partition $\Omega_h$ of the domain $\Omega$ is
$$
\Omega_h=\{(x_i,y_j)|x_{i}=ih,y_j=jh,0\leq i,j\leq N\},
$$
with the periodic grid function spaces
\begin{align*}
\mathcal{C}:=\{u\,|\,u_{i,j}=u_{i\pm N,j\pm N},\forall\, i,j=1,2,...,N\}.
\end{align*}

The average and difference operators can be defined as
\begin{align*}
&A_xu_{i+\frac12,j}=\frac{u_{i+1,j}+u_{i,j}}{2},\quad A_yu_{i,j+\frac12}=\frac{u_{i,j+1}+u_{i,j}}{2},\\
&D_xu_{i+\frac12,j}=\frac{u_{i+1,j}-u_{i,j}}{h},\quad D_yu_{i,j+\frac12}=\frac{u_{i,j+1}-u_{i,j}}{h},\\
&a_xu_{i,j}=\frac{u_{i+\frac12,j}+u_{i-\frac12,j}}{2},\quad a_yu_{i,j}=\frac{u_{i,j+\frac12}+u_{i,j-\frac12}}{2},\\
&d_xu_{i,j}=\frac{u_{i+\frac12,j}-u_{i-\frac12,j}}{h},\quad d_yu_{i,j}=\frac{u_{i,j+\frac12}-u_{i,j-\frac12}}{h}.
\end{align*}
Then, for  grid functions $u,v\in\mathcal{C}$, the discrete gradient and discrete divergence operators can be denoted by
\begin{align*}
\nabla_h u_{i,j}=(D_xu_{i+\frac12,j},D_yu_{i,j+\frac12}),\qquad \nabla_h\cdot (u,v)_{i,j}=d_xu_{i,j}+d_yv_{i,j},
\end{align*}
and the discrete Laplacian $\Delta_h $ can be given by
\begin{align*}
\Delta_h u_{i,j}&=\nabla_h\cdot(\nabla_h u)_{i,j}=d_x(D_x u)_{i,j}+d_y(D_yu_{i,j})\\
&=\frac{1}{h^2}(u_{i+1,j}+u_{i-1,j}+u_{i,j+1}+u_{i,j-1}-4u_{i,j}).
\end{align*}
Moreover, if $\mathcal{D}$ is a scalar function defined at the central point, we have
\begin{align}\label{hlapl}
\nabla_h\cdot(\mathcal{D}\nabla_h u)_{i,j}=d_x(\mathcal{D}D_x u)_{i,j}+d_y(\mathcal{D}D_y u)_{i,j}.
\end{align}
Obviously, the above discrete operators are all second-order approximations of the corresponding differential operators. 

We also recall the discrete $L^2$ inner product $\langle u,v\rangle_h=h^2\sum\limits_{i,j=1}^N u_{i,j}v_{i,j},$
with the induced norm $\|u\|_h=\langle u,u\rangle_h^\frac12$, the discrete $H^1$ inner product
\begin{align*}
\langle \nabla_h u,\nabla_h v\rangle_h=h^2\sum_{i,j=1}^N (D_xu_{i-1/2,j}D_xv_{i-1/2,j}+D_yu_{i,j-1/2}D_yv_{i,j-1/2}),
\end{align*}
with the induced norm $\|\nabla_h u\|_h=\langle \nabla_hu,\nabla_hu\rangle_h^\frac12$ and the infinity norm $\|u\|_\infty:=\max\limits_{1\leq i,j\leq N}|u_{i,j}|$.

For a given scalar function $\phi$, let us define a linear elliptic differential operator $\mathcal{L}[\phi]$ as 
\begin{align*}
\mathcal{L}[\phi] u=\nabla\cdot(e^{\phi}\nabla\frac{u}{e^{\phi}}).
\end{align*}
Here, $\frac{u}{e^\phi}$ are called as the Slotboom variable. The resulting elliptic operator is of a generalized Fokker-Planck form and can be discretized with symmetrical fluxes, often yielding a mass-conservative and positivity-preserving scheme \cite{DWZ19,HPY19}. Using \eqref{hlapl}, the symmetric discretization of $\mathcal{L}[\phi]$, denoted by $\mathcal{L}_h[\phi]$, is given by
\begin{align}\label{sym}
\begin{aligned}
(\mathcal{L}_h[\phi]u)_{i,j}:=&[\nabla_h\cdot(\overline{e^{\phi}}\nabla_h\frac{u}{e^{\phi}})]_{i,j}=d_x(\overline{e^{\phi}}D_x(\frac{u}{e^{\phi}}))_{i,j}+d_y(\overline{e^{\phi}}D_y(\frac{u}{e^{\phi}}))_{i,j}\\
  =&\frac{1}{h}\left(\overline{e^{\phi_{i+\frac12,j}}}\left(\frac{u_{i+1,j}}{e^{\phi_{i+1,j}}}-\frac{u_{i,j}}{e^{\phi_{i,j}}}\right)+\overline{e^{\phi_{i-\frac12,j}}}\left(\frac{u_{i-1,j}}{e^{\phi_{i-1,j}}}-\frac{u_{i,j}}{e^{\phi_{i,j}}}\right)\right)\\
  &+\frac{1}{h}\left(\overline{e^{\phi_{i,j+\frac12}}}\left(\frac{u_{i,j+1}}{e^{\phi_{i,j+1}}}-\frac{u_{i,j}}{e^{\phi_{i,j}}}\right)+\overline{e^{\phi_{i,j-\frac12}}}\left(\frac{u_{i,j-1}}{e^{\phi_{i,j-1}}}-\frac{u_{i,j}}{e^{\phi_{i,j}}}\right)\right)
  \end{aligned}
  \end{align}
where $\overline{e^{\phi}}$ is the harmonic mean approximation of the central point value, 
\begin{align*}
\begin{aligned}
\overline{e^{\phi_{i\pm\frac12,j}}}=&\left(\frac{e^{-\phi_{i\pm1,j}}+e^{-\phi_{i,j}}}{2}\right)^{-1}=\frac{2e^{\phi_{i\pm 1,j}}e^{\phi_{i,j}}}{e^{\phi_{i\pm1,j}}+e^{\phi_{i,j}}},\, \overline{e^{\phi_{i,j\pm\frac12}}}=&\left(\frac{e^{-\phi_{i,j\pm1}}+e^{-\phi_{i,j}}}{2}\right)^{-1}=\frac{2e^{\phi_{i,j\pm 1}}e^{\phi_{i,j}}}{e^{\phi_{i,j\pm1}}+e^{\phi_{i,j}}}.
 \end{aligned}
\end{align*}
It is easy to verify that $\mathcal{L}_h[\phi]$ is the second-order approximation of $ \mathcal{L}_h[\phi]$.
\begin{remark}
The central point value of $e^{\phi}$ can also be approximated by the geometric, arithmetic and entropy mean. No matter which mean is used, the mass conservation, positivity preservation and energy stability can be proven in the fully discrete settings.
\end{remark}
\begin{lemma}\label{lem_conservation}
For the given functions $\phi\in\mathcal{C}$ and $u\in\mathcal{C}$, it holds that $\langle \mathcal{L}_h[\phi]u,\mathbf{1}\rangle_h=0$.
\end{lemma}
\begin{proof}
It follows from \eqref{sym} that 
\begin{align}\label{sum}
\begin{aligned}
\langle \mathcal{L}_h[\phi]u,\mathbf{1}\rangle_h=&h^2\sum_{i=0}^N\sum_{j=0}^N (\mathcal{L}_h[\phi]u)_{i,j}\\
=&2h\sum_{i=0}^N\sum_{j=0}^N\left(\frac{e^{\phi_{i+1,j}}e^{\phi_{i,j}}}{e^{\phi_{i+1,j}}+e^{\phi_{i,j}}}\left(\frac{u_{i+1,j}}{e^{\phi_{i+1,j}}}-\frac{u_{i,j}}{e^{\phi_{i,j}}}\right)+\frac{e^{\phi_{i-1,j}}e^{\phi_{i,j}}}{e^{\phi_{i-1,j}}+e^{\phi_{i,j}}}\left(\frac{u_{i-1,j}}{e^{\phi_{i-1,j}}}-\frac{u_{i,j}}{e^{\phi_{i,j}}}\right)\right)\\
  &+2h\sum_{i=0}^N\sum_{j=0}^N\left(\frac{e^{\phi_{i,j+1}}e^{\phi_{i,j}}}{e^{\phi_{i,j+1}}+e^{\phi_{i,j}}}\left(\frac{u_{i,j+1}}{e^{\phi_{i,j+1}}}-\frac{u_{i,j}}{e^{\phi_{i,j}}}\right)+\frac{e^{\phi_{i,j-1}}e^{\phi_{i,j}}}{e^{\phi_{i,j-1}}+e^{\phi_{i,j}}}\left(\frac{u_{i,j-1}}{e^{\phi_{i,j-1}}}-\frac{u_{i,j}}{e^{\phi_{i,j}}}\right)\right).
  \end{aligned}
\end{align}
According to the periodic boundary condition, the first term of the right hand side of \eqref{sum} is
\begin{align*}
\begin{aligned}
&2h\sum_{i=0}^N\sum_{j=0}^N\left(\frac{e^{\phi_{i+1,j}}e^{\phi_{i,j}}}{e^{\phi_{i+1,j}}+e^{\phi_{i,j}}}\left(\frac{u_{i+1,j}}{e^{\phi_{i+1,j}}}-\frac{u_{i,j}}{e^{\phi_{i,j}}}\right)+\frac{e^{\phi_{i-1,j}}e^{\phi_{i,j}}}{e^{\phi_{i-1,j}}+e^{\phi_{i,j}}}\left(\frac{u_{i-1,j}}{e^{\phi_{i-1,j}}}-\frac{u_{i,j}}{e^{\phi_{i,j}}}\right)\right)\\
=&2h\sum_{j=0}^N\left(\frac{e^{\phi_{N+1,j}}e^{\phi_{N,j}}}{e^{\phi_{N+1,j}}+e^{\phi_{N,j}}}\left(\frac{u_{N+1,j}}{e^{\phi_{N+1,j}}}-\frac{u_{N,j}}{e^{\phi_{N,j}}}\right)-\frac{e^{\phi_{-1,j}}e^{\phi_{0,j}}}{e^{\phi_{-1,j}}+e^{\phi_{0,j}}}\left(\frac{u_{-1,j}}{e^{\phi_{-1,j}}}-\frac{u_{0,j}}{e^{\phi_{0,j}}}\right)\right)=0.
\end{aligned}
\end{align*}
Similarly, the second term of the right side of \eqref{sum} is also zero. The proof is completed.
\end{proof}
\begin{lemma}\label{lem_egy}
For the given functions $\phi\in\mathcal{C}$ and $u\in\mathcal{C}$, if $u>0$, then $\langle \mathcal{L}_h[\phi]u,\ln \frac{u}{e^\phi}\rangle_h\leq 0$. 
\end{lemma}
\begin{proof}
It follows from \eqref{sym} that 
\begin{align*}
\begin{aligned}
\langle \mathcal{L}_h[\phi]u,\ln \frac{u}{e^\phi}\rangle_h=&h^2\sum_{i=0}^N\sum_{j=0}^N (\mathcal{L}_h[\phi]u)_{i,j}\ln \frac{u_{i,j}}{e^{\phi_{i,j}}}\\
=& h^2\sum_{i=0}^N\sum_{j=0}^N [\nabla_h\cdot(\overline{e^{\phi}}\nabla_h\frac{u}{e^{\phi}})]_{i,j}\ln \frac{u_{i,j}}{e^{\phi_{i,j}}}\\
=&-2h\sum_{i=0}^N\sum_{j=0}^N\frac{e^{\phi_{i+1,j}}e^{\phi_{i,j}}}{e^{\phi_{i+1,j}}+e^{\phi_{i,j}}}\left(\frac{u_{i+1,j}}{e^{\phi_{i+1,j}}}-\frac{u_{i,j}}{e^{\phi_{i,j}}}\right)\left(\ln\frac{u_{i+1,j}}{e^{\phi_{i+1,j}}}-\ln\frac{u_{i,j}}{e^{\phi_{i,j}}}\right)\\
  &-2h\sum_{i=0}^N\sum_{j=0}^N\frac{e^{\phi_{i,j+1}}e^{\phi_{i,j}}}{e^{\phi_{i,j+1}}+e^{\phi_{i,j}}}\left(\frac{u_{i,j+1}}{e^{\phi_{i,j+1}}}-\frac{u_{i,j}}{e^{\phi_{i,j}}}\right)\left(\ln\frac{u_{i,j+1}}{e^{\phi_{i,j+1}}}-\ln\frac{u_{i,j}}{e^{\phi_{i,j}}}\right)\\
  \leq& 0,
  \end{aligned}
\end{align*}
where in the last inequality we have used that $(x-y)(\ln x-\ln y)\geq 0$ for any $x,y>0$, which completes the proof.
\end{proof}
\begin{lemma}\label{lem_M}
Let $A$ be the coefficient matrix resulting from the numerical discretization operator $-\mathcal{L}_h[\phi]$ for a given grid function $\phi\in \mathcal{C}$, then $A$ is an $M$-matrix and thus $A^+\geq 0$, where $A^+$ is the generalized Moore-Penrose inverse of $A$. 
\end{lemma}
\begin{proof}
Following the idea in \cite{BCV14,DWZ19}, we go through the nonzero entries in each column of matrix $A$. Defining $$[i,j]=(i-1)N+j, \mbox{ for } i,j=1,2,\dots,N,$$
and then non-zero entries of the $l$-th column ($l=[i,j]$) are given by
  \begin{align*}
  A_{k,l}=\frac{2}{h^2}\left\{
  \begin{array}{ll}
  \displaystyle -\frac{1}{1+e^{\phi_{i,j}-\phi_{i-1,j}}},& k=[\mbox{mod}(i-1,N),j],\\ 
  \displaystyle -\frac{1}{1+e^{\phi_{i,j}-\phi_{i,j-1}}},& k=[i,\mbox{mod}(j-1,N)],\\
  \displaystyle \displaystyle\frac{1}{1+e^{\phi_{i,j}-\phi_{i,j+1}}}+\frac{1}{1+e^{\phi_{i,j}-\phi_{i+1,j}}}+\frac{1}{1+e^{\phi_{i,j}-\phi_{i,j-1}}}+\frac{1}{1+e^{\phi_{i,j}-\phi_{i-1,j}}},& k=l,\\
  \displaystyle -\frac{1}{1+e^{\phi_{i,j}-\phi_{i,j+1}}},& k=[i,\mbox{mod}(j+1,N)],\\
 \displaystyle -\frac{1}{1+e^{\phi_{i,j}-\phi_{i+1,j}}}, & k=[\mbox{mod}(i+1,N),j],\\
  \end{array}
  \right.
  \end{align*}
  where $\mbox{mod}(k,l)$ returns the remainder of $k$ divided by $l$.
 Furthermore, we observe the following property
\begin{align*}
\left\{
\begin{array}{ll}
\sum\limits_{k=1}^{N^2} A_{k,l}= 0, &\mbox{ for } l=1,2,\cdots,N^2,\\
A_{k,l}>0, &\mbox{ for } l=1,2,\cdots,N^2,\\
A_{k,l}\leq 0, &\mbox{ for } k,l=1,2,\cdots,N^2, \mbox{ and } k\neq l.\\
\end{array}
\right.
\end{align*}
Additionally, we have
\begin{align*}
A_{l,l}= \sum\limits_{k=1,k\neq l}^{N^2} |A_{k,l}|,\mbox{ for } l=1,2,\cdots,N^2.
\end{align*}
Thus the matrix $A$ has positive diagonal terms and nonpositive offdiagonal terms and
is diagonally dominant with respect to its columns. This means that $A$ is a singular $M$-matrix and its generalized inverse has only nonnegative coefficients \cite{BP94}. The proof is completed. 
\end{proof}
\subsection{Space discrete form}
The Possion-Nernst-Planck equations \eqref{PNP} can be thus written in an equivalent form 
\begin{subequations}\label{sPNP}
\begin{align}
p_t =& \mathcal{L}[-\phi] p, \label{sPNP_p}\\
n_t =& \mathcal{L}[\phi] n,\label{sPNP_n}\\
\phi =&(-\epsilon^2\Delta_h)^{-1}(p-n+\rho^f).\label{sPNP_phi}
\end{align}
\end{subequations}
The discrete form of \eqref{sPNP} is
\begin{subequations}\label{dPNP}
\begin{align}
p_t =& \mathcal{L}_h[-\phi]p, \label{dPNP_p}\\
n_t =& \mathcal{L}_h[\phi]n,\label{dPNP_n}\\
\phi =&(-\epsilon^2\Delta_h)^{-1}(p-n+\rho^f).\label{dPNP_phi}
\end{align}
\end{subequations}
\begin{remark}
To ensure the uniqueness of the electric potential, mean zero condition is required for $\phi$,
\begin{align*}
\langle \phi,1\rangle_h = 0.
\end{align*}     
\end{remark}


\section{Temporal discretization}\label{sec:temporal}
In this section, we construct the temporal discretization for the PNP equations by using the ETD approach. We will start with the space discrete form \eqref{dPNP} to develop the structure preserving ETD schemes. At this point, we rewrite the spatial discrete system \eqref{dPNP} as
 \begin{subequations}\label{ETD}
\begin{align}
p_s(t+s) =& \mathcal{L}_h[-\phi(t+s)]p(t+s),\\
n_s(t+s) =& \mathcal{L}_h[\phi(t+s)]n(t+s),\\
\phi(t+s) =& (-\epsilon^2\Delta_h)^{-1}(p(t+s)-n(t+s)+\rho^f),
\end{align}
\end{subequations} 
for any $t\geq 0$ and $s>0$.
 
\subsection{First order ETD scheme}
Let $\tau=T/K_t$  $(K_t\in\mathbb{Z}^+)$ be a uniform time step and $t_k=k\tau (n=0,1,2,...,K_t)$.
Denote by $u^n$ the discrete approximation of $u(t_n)$. Setting $t=t_k$ and $s\in(0,\tau]$ in \eqref{ETD} gives
 \begin{subequations}\label{ETD1_org}
\begin{align}
p_s(t_k+s) =& \mathcal{L}_h[-\phi(t_k+s)]p(t_k+s), \\
n_s(t_k+s) =& \mathcal{L}_h[\phi(t_k+s)]n(t_k+s),\\
\phi(t_k+s) =& (-\epsilon^2\Delta_h)^{-1}(p(t_k+s)-n(t_k+s)+\rho^f).
\end{align}
\end{subequations} 

Letting $\mathcal{L}_h[\pm\phi(t_k+s)]\approx \mathcal{L}_h[\pm\phi(t_k)]$ in \eqref{ETD1_org}, we obtain the first order ETD (ETD1) scheme of \eqref{ETD}: for $k\geq 0$, find $p^{k+1}=q^k(\tau)$, $n^{k+1}=m^k(\tau)$ and $\psi^{k+1}=\phi^k(\tau)$ by solving
\begin{subequations}\label{ETD1}
\begin{align}
q^k_s =& \mathcal{L}_h[-\phi^k]q^k,\\
m^k_s =& \mathcal{L}_h[\phi^k]m^k,\\
\psi^k(s) =& (-\epsilon^2\Delta_h)^{-1}(q^k(s)-m^k(s)+\rho^f),
\end{align}
\end{subequations} 
with the initial value $q^k(0)=p^k$, $m^k(0)=n^k$, whose solutions satisfy
 \begin{subequations}\label{rETD1}
\begin{align}
p^{k+1} = &\varphi_0(\tau\mathcal{L}_h[-\phi^k])p^k,\\
n^{k+1} = &\varphi_0(\tau\mathcal{L}_h[\phi^k])n^k,\\
\phi^{k+1} =& (-\epsilon^2\Delta_h)^{-1} (p^{k+1}-n^{k+1}+\rho^f),
\end{align}
\end{subequations} 
where $\varphi_0(a)=e^a$.
\begin{theorem}[Mass conservation of ETD1]\label{thm_ETD1_mass}
The solutions of ETD1 scheme \eqref{rETD1} satisfy the mass conservation unconditionally, that is, for any $\tau>0$, ETD1 solutions satisfy
\begin{align*}
\langle p^{k+1},1\rangle_h =\langle p^{k},1\rangle_h = \cdots =\langle p^0,1\rangle_h:= M_p,
\end{align*}
and
\begin{align*}
\langle n^{k+1},1\rangle_h =\langle n^{k},1\rangle_h = \cdots =\langle n^0,1\rangle_h:= M_n.
\end{align*}
\end{theorem}
\begin{proof}
Let us first prove that $\langle p^{k},1\rangle_h=M_p$ implies that $\langle p^{k+1},1\rangle_h=M_p$. Taking $L^2$ inner product with 1 on both side of \eqref{ETD1} and using Lemma \ref{lem_conservation}, we have 
\begin{align*}
&\frac{d}{ds}\langle q^k(s),1\rangle_h =0,
\end{align*}
which implies that the quantity $V_p^k(s)=\langle q^k(s),1\rangle_h$ satisfies the ODE
\begin{align*}
&\frac{dV_p^k(s)}{ds}=0, \,V_p^k(0)=M_p,
\end{align*}
whose solution is $V_p^k(\tau)=M_p$. Thus we have $\langle p^{k+1},1\rangle_h=M_p$. 

Repeating the similar process, one can also obtain $\langle n^{k},1\rangle_h=M_n$ implies that $\langle n^{k+1},1\rangle_h=M_n$, which completes the proof.
\end{proof}

\begin{theorem}[Positivity preservation of ETD1]\label{thm_ETD1_mp}
The solutions of ETD1 scheme \eqref{rETD1} satisfy the positivity preservation unconditionally, that is, for any $\tau> 0$, if the initial value satisfies $p^{0}\geq 0\, \mbox{ and }\, n^{0}\geq 0$, then ETD1 solutions satisfy
\begin{align*}
p^{k}\geq 0\, \mbox{ and }\, n^{k}\geq 0.
\end{align*}
\end{theorem} 
\begin{proof}
It suffices to prove $p^{k+1}\geq 0$ if $p^{k}\geq 0$. Under the periodic boundary condition, the fully discrete ETD1 scheme \eqref{ETD1} can be expressed in a matrix form
\begin{align}\label{ETD1_matrix}
&q^k_s + Bq^k =0,
\end{align}
where $B$ is the coefficient matrix resulting from the numerical discretization operator $-\mathcal{L}_h[-\phi^k]$ related to the variable $\phi^{k}$. 

To prove $p^{k+1}\geq 0$, we first suppose that
\begin{align*}
q^{k}(s^*)=\min_{0\leq s\leq \tau} q^{k}(s),
\end{align*}
which implies that $q^{k}_s(s^*)\leq 0$. It follows from \eqref{ETD1_matrix} that
\begin{align*}
Bq^k(s^*) \geq 0.
\end{align*}
According to Lemma \ref{lem_M}, it is easy to verify that $B$ is a $M$-matrix which implies that $q^k(s^*)\geq 0$ and thus $p^{k+1}\geq 0$.


Repenting the same process, one can obtain  $n^{k+1}\geq 0$ if $n^{k}\geq 0$, which completes the proof.
\end{proof}

\subsection{Second order ETD scheme}
Next we consider the second-order approximation of \eqref{ETD} by setting $t=t_{k-1}$ and $s\in (0,2\tau]$,
 \begin{subequations}\label{ETD2_org}
\begin{align}
p_s(t_{k-1}+s) =& \mathcal{L}_h[-\phi(t_{k-1}+s)]p(t_{k-1}+s),\\
n_s(t_{k-1}+s) =& \mathcal{L}_h[\phi(t_{k-1}+s)]n(t_{k-1}+s),\\
\phi(t_{k-1}+s) =& (-\epsilon^2\Delta_h)^{-1} (p(t_{k-1}+s)-n(t_{k-1}+s)+\rho^f).
\end{align}
\end{subequations} 

Using the approximation $\mathcal{L}_h[\pm\phi(t_k+s)]\approx \mathcal{L}_h[\pm\phi(t_k)]$ in \eqref{ETD2_org}, we obtain the following second-order ETD (ETD2) scheme of \eqref{ETD}: for $k\geq 1$, find $p^{k+1}=q^{k-1}(2\tau)$, $n^{k+1}=m^{k-1}(2\tau)$ and $\psi^{k+1}=\phi^{k-1}(2\tau)$ by solving
\begin{subequations}\label{ETD2}
\begin{align}
q^{k-1}_s = &\mathcal{L}_h[-\phi^k]q^{k-1},\\
m^{k-1}_s = &\mathcal{L}_h[\phi^k]m^{k-1},\\
\psi^{k-1}(s) =& (-\epsilon^2\Delta_h)^{-1} (q^{k-1}(s)-m^{k-1}(s)+\rho^f), 
\end{align}
\end{subequations}
with the initial value $q^{k-1}(0)=p^{k-1}$, $m^{k-1}(0)=n^{k-1}$, whose solutions satisfy
 \begin{subequations}\label{rETD2}
\begin{align}
p^{k+1} =& \varphi_0(2\tau\mathcal{L}_h[-\phi^k])p^{k-1},\\
n^{k+1} =& \varphi_0(2\tau\mathcal{L}_h[\phi^k])n^{k-1},\\
\phi^{k+1} =& (-\epsilon^2\Delta_h)^{-1} (p^{k+1}-n^{k+1}+\rho^f).
\end{align}
\end{subequations} 
Since \eqref{rETD2} is a three-level scheme, for the first time step, we adopt the ETD1 scheme to obtain $p^1$ and $n^1$.
\begin{theorem}[Mass conservation of ETD2]\label{thm_ETD2_mass}
The solutions of ETD2 scheme \eqref{rETD2} satisfy the mass conservation unconditionally, that is, for any $\tau>0$, ETD2 solutions satisfy
\begin{align*}
\langle p^{k+1},1\rangle_h =\langle p^{k},1\rangle_h = \cdots =\langle p^0,1\rangle_h = M_p,
\end{align*}
and
\begin{align*}
\langle n^{k+1},1\rangle_h =\langle n^{k},1\rangle_h = \cdots =\langle n^0,1\rangle_h = M_n.
\end{align*}
\begin{proof}
Following the same proof in Theorem \ref{thm_ETD1_mass}, taking $L^2$ inner product with 1 on both side of \eqref{rETD2} and using Lemma \ref{lem_conservation} gives 
\begin{align*}
&\frac{d}{ds}\langle q^{k-1}(s),1\rangle_h=0,
\end{align*}
whose solution is $\langle q^{k-1}(2\tau),1\rangle_h=M_p$. Thus we have $\langle p^{k+1},1\rangle_h=M_p$. One can also obtain that $\langle n^{k+1},1\rangle_h=M_n$ if $\langle n^k,1\rangle_h=M_n$, which completes the proof. 
\end{proof}
\end{theorem}
\begin{theorem}[Positivity preservation of ETD2]
The solutions of ETD2 scheme \eqref{ETD2} satisfy the positivity preservation unconditionally, that is, for any $\tau> 0$, if the initial value satisfies $p^{0}\geq 0\, \mbox{ and }\, n^{0}\geq 0$, then ETD2 solutions satisfy
\begin{align*}
p^{k}\geq 0\, \mbox{ and }\, n^{k}\geq 0.
\end{align*}
\end{theorem} 
\begin{proof}
It suffices to prove $p^{k+1}\geq 0$ if $p^{k}\geq 0$. Under the periodic boundary condition, the fully discrete ETD2 scheme \eqref{ETD2} can be expressed in a matrix form
\begin{align}\label{ETD2_matrix}
&q^{k-1}_s + Bq^{k-1} =0.
\end{align} 
where $B$ is defined in \eqref{ETD1_matrix}.

To prove $p^{k+1}\geq 0$, we first suppose that
\begin{align*}
q^{k-1}(s^*)=\min_{0\leq s\leq 2\tau} q^{k-1}(s),
\end{align*}
which implies that $q^{k-1}_s(s^*)\leq 0$. Thus from \eqref{ETD2_matrix} we have 
\begin{align*}
Bq^{k-1}(s^*) \geq 0.
\end{align*}
Since $B$ is a $M$-matrix, we obtain $q^k(s^*)\geq 0$ and thus $p^{k+1}\geq 0$.

Repeating the same process, one can obtain  $n^{k+1}\geq 0$ if $n^{k}\geq 0$, which completes the proof.
\end{proof}
\section{Energy stability}\label{sec:egy}
In this section, we will prove that the proposed ETD schemes satisfy the energy stability with respect to the given free energy functional defined in \eqref{egy}.
We begin to state the energy stability of ETD1 scheme \eqref{ETD1}.
\begin{theorem}[Energy stability of ETD1]
If the initial ion concentrations satisfy $p_0\geq 0\, \mbox{ and }\, n_0\geq 0$, then solutions of ETD1 scheme \eqref{ETD1} satisfy the energy stability, that is, for any $k\geq 0$,
\begin{align*}
E^{k+1}\leq E^k +\frac{\epsilon^2}{2}\|\nabla_h(\phi^{k+1}-\phi^k)\|_h^2, 
\end{align*}
where $E^k$ is the discrete version of continuous free energy $E$, defined by
\begin{align*}
E^k=\langle p^k\ln p^k,1\rangle_h+\langle n^k\ln n^k,1\rangle_h+\frac{\epsilon^2}{2}\|\nabla_h\phi^k\|_h^2.
\end{align*}
\end{theorem}
\begin{proof}
Taking $L^2$ inner product with $\ln \frac{q^k}{e^{-\phi^k}}$ on both side of \eqref{ETD1} and using Lemma \ref{lem_egy}, we have 
\begin{align*}
\langle q^k_s,\ln \frac{q^k}{e^{-\phi^k}}\rangle_h =& \langle\mathcal{L}_{h}[-\phi^k]q^k,\ln \frac{q^k}{e^{-\phi^k}}\rangle_h\leq 0,
\end{align*}
which implies the quantity $W^k(s)=\langle q^k\ln q^k-q^k,1\rangle_h+\langle q^k,\phi^k\rangle_h$ satisfies the ODE
\begin{align}\label{ETD2_ODE}
\begin{aligned}
&\frac{dW^k(s)}{ds}\leq 0,\,s\in(0,\tau],\\
&W^k(0)= \langle p^k\ln p^k-p^k,1\rangle_h.
\end{aligned}
\end{align}
Thus we have from \eqref{ETD2_ODE} that $W^k(\tau)\leq W^k(0)$ which implies
\begin{align*}
\langle p^{k+1}\ln p^{k+1}-p^{k+1},1\rangle_h+\langle p^{k+1},\phi^k\rangle_h\leq \langle p^k\ln p^k-p^k,1\rangle_h+\langle p^k,\phi^k\rangle_h.
\end{align*}
Similarly, one can obtain
\begin{align*}
\langle n^{k+1}\ln n^{k+1}-n^{k+1},1\rangle_h-\langle n^{k+1},\phi^k\rangle_h\leq \langle n^k\ln n^k-n^k,1\rangle_h-\langle n^k,\phi^k\rangle_h.
\end{align*}
Combining with the above inequalities gives
\begin{align}\label{ETD1_egy_pn}
\begin{aligned}
&\langle p^{k+1}\ln p^{k+1},1\rangle_h+\langle n^{k+1}\ln n^{k+1},1\rangle_h+\langle p^{k+1}-n^{k+1}-p^k+n^k,\phi^k\rangle_h\\
\leq &\langle p^k\ln p^k,1\rangle_h+\langle n^k\ln n^k,1\rangle_h+\langle p^{k+1}+n^{k+1},1\rangle_h-\langle p^k+n^k,1\rangle_h\\
\leq &\langle p^k\ln p^k,1\rangle_h+\langle n^k\ln n^k,1\rangle_h,
\end{aligned}
\end{align}
where we have used Theorem \ref{thm_ETD1_mass}.

It follows from \eqref{ETD1} that 
\begin{align*}
p^{k+1}-n^{k+1}-(p^k-n^k)=&p^{k+1}-n^{k+1}+\rho^f-(p^k-n^k+\rho^f)=-\epsilon^2\Delta_h(\phi^{k+1}-\phi^k)
\end{align*}
which implies that 
\begin{align}\label{ETD1_egy_phi}
\begin{aligned}
\langle p^{k+1}-n^{k+1}-p^k+n^k,\phi^k\rangle_h=&-\epsilon^2\langle \Delta_h(\phi^{k+1}-\phi^k),\phi^k\rangle_h=\epsilon^2\langle\nabla_h(\phi^{k+1}-\phi^k),\nabla_h\phi^k\rangle_h\\
=&-\frac{\epsilon^2}{2}(\|\nabla_h\phi^k\|_h^2-\|\nabla_h\phi^{k+1}\|_h^2+\|\nabla_h(\phi^{k+1}-\phi^k)\|_h^2).
\end{aligned}
\end{align}
Finally, putting \eqref{ETD1_egy_phi} into \eqref{ETD1_egy_pn} gives 
\begin{align*}
\begin{aligned}
&\langle p^{k+1}\ln p^{k+1},1\rangle_h+\langle n^{k+1}\ln n^{k+1},1\rangle_h+\frac{\epsilon^2}{2}\|\nabla_h\phi^{k+1}\|_h^2\\
\leq &\langle p^k\ln p^k,1\rangle_h+\langle n^k\ln n^k,1\rangle_h+\frac{\epsilon^2}{2}(\|\nabla_h\phi^k\|_h^2+\|\nabla_h(\phi^{k+1}-\phi^k)\|_h^2),
\end{aligned}
\end{align*}
which completes the proof. 
\end{proof}

Next, we turn to the energy stability of ETD2 scheme \eqref{ETD2}.
\begin{theorem}[Energy stability of ETD2]\label{thm_ETD2_egy}
If the initial ion concentrations satisfy $p_0\geq 0\, \mbox{ and }\, n_0\geq 0$, then solutions of ETD2 scheme \eqref{ETD2} satisfy the energy stability, that is, for any $k\geq 1$,
\begin{align*}
&\widehat{E}^{k+1}\leq \widehat{E}^{k},
\end{align*}
where 
\begin{align*}
\widehat{E}^{k}=\frac12(\langle p^k\ln p^k,1\rangle_h+\langle n^k\ln n^k,1\rangle_h+\langle p^{k-1}\ln p^{k-1},1\rangle_h+\langle n^{k-1}\ln n^{k-1},1\rangle_h)+\frac{\epsilon^2}{2}\langle\nabla_h\phi^k,\nabla_h\phi^{k-1}\rangle_h.
\end{align*}
\end{theorem}
\begin{proof}
Taking $L^2$ inner product with $\ln \frac{q^{k-1}}{e^{-\phi^k}}$ on both side of \eqref{ETD2} and using Lemma \ref{lem_egy}, we have 
\begin{align*}
\langle q^{k-1}_s,\ln \frac{q^{k-1}}{e^{-\phi^k}}\rangle_h =& \langle\mathcal{L}_{h}[-\phi^k]q^{k-1},\ln \frac{q^{k-1}}{e^{-\phi^k}}\rangle_h\leq 0,
\end{align*}
which implies the quantity $W^{k-1}(s)=\langle q^{k-1}\ln q^{k-1}-q^{k-1},1\rangle_h+\langle q^{k-1},\phi^k\rangle_h$ satisfies the ODE
\begin{align}\label{ETD2_egy_ODE}
\begin{aligned}
&\frac{dW^{k-1}(s)}{ds}\leq 0,\\
&W^{k-1}(0)= \langle p^{k-1}\ln p^{k-1}-p^{k-1},1\rangle_h+\langle p^{k-1},\phi^k\rangle_h.
\end{aligned}
\end{align}
The solution of the above ODE \eqref{ETD2_egy_ODE} is $W^{k-1}(2\tau)\leq W^{k-1}(0)$. Thus we have
 \begin{align*}
&\langle p^{k+1}\ln p^{k+1}-p^{k+1},1\rangle_h+\langle p^{k+1},\phi^k\rangle_h\leq \langle p^{k-1}\ln p^{k-1}-p^{k-1},1\rangle_h+\langle p^{k-1},\phi^k\rangle_h.
\end{align*}
Similarly, one can obtain
\begin{align*}
&\langle n^{k+1}\ln n^{k+1}-n^{k+1},1\rangle_h+\langle n^{k+1},\phi^k\rangle_h
\leq \langle n^{k-1}\ln n^{k-1}-n^{k-1},1\rangle_h-\langle n^{k-1},\phi^k\rangle_h.
\end{align*}
Combining with the above inequalities and using Theorem \ref{thm_ETD2_mass} gives
\begin{align}\label{ETD2_egy_pn}
\begin{aligned}
&\langle p^{k+1}\ln p^{k+1},1\rangle_h+\langle n^{k+1}\ln n^{k+1},1\rangle_h+\langle p^{k+1}-n^{k+1},\phi^k\rangle_h\\
\leq &\langle p^{k-1}\ln p^{k-1},1\rangle_h+\langle n^{k-1}\ln n^{k-1},1\rangle_h+\langle p^{k-1}-n^{k-1},\phi^k\rangle_h.
\end{aligned}
\end{align}

It follows from \eqref{ETD2} that 
\begin{align}\label{ETD2_egy_phi}
\begin{aligned}
\langle p^{k+1}-n^{k+1},\phi^k\rangle_h=-\epsilon^2\langle \Delta_h\phi^{k+1},\phi^k\rangle_h-\langle\rho^f,\phi^k\rangle_h=\epsilon^2\langle\nabla_h\phi^{k+1},\nabla_h\phi^k\rangle_h-\langle\rho^f,\phi^k\rangle_h,\\
\langle p^{k-1}-n^{k-1},\phi^k\rangle_h=-\epsilon^2\langle \Delta_h\phi^{k-1},\phi^k\rangle_h-\langle\rho^f,\phi^k\rangle_h=\epsilon^2\langle\nabla_h\phi^{k-1},\nabla_h\phi^k\rangle_h-\langle\rho^f,\phi^k\rangle_h.
\end{aligned}
\end{align}
Finally, putting \eqref{ETD2_egy_phi} into \eqref{ETD2_egy_pn} gives 
\begin{align}\label{energy_result}
\begin{aligned}
&\langle p^{k+1}\ln p^{k+1},1\rangle_h+\langle n^{k+1}\ln n^{k+1},1\rangle_h+\epsilon^2\langle\nabla_h\phi^{k+1},\nabla_h\phi^k\rangle_h\\
\leq &\langle p^{k-1}\ln p^{k-1},1\rangle_h+\langle n^{k-1}\ln n^{k-1},1\rangle_h+\epsilon^2\langle\nabla_h\phi^{k-1},\nabla_h\phi^k\rangle_h.
\end{aligned}
\end{align}
Adding $\langle p^k\ln p^k,1\rangle_h+\langle n^k\ln n^k,1\rangle_h$ in both side of the above inequality yields
\begin{align*}
&\frac12(\langle p^{k+1}\ln p^{k+1},1\rangle_h+\langle n^{k+1}\ln n^{k+1},1\rangle_h+\langle p^k\ln p^k,1\rangle_h+\langle n^k\ln n^k,1\rangle_h)+\frac{\epsilon^2}{2}\langle\nabla_h\phi^{k+1},\nabla_h\phi^k\rangle_h\\
\leq &\frac12(\langle p^k\ln p^k,1\rangle_h+\langle n^k\ln n^k,1\rangle_h+\langle p^{k-1}\ln p^{k-1},1\rangle_h+\langle n^{k-1}\ln n^{k-1},1\rangle_h)+\frac{\epsilon^2}{2}\langle\nabla_h\phi^k,\nabla_h\phi^{k-1}\rangle_h,
\end{align*}
which leads to the desired result. The proof is completed.
 \end{proof}

\section{Numerical experiments}\label{sec:num}
In this section, we will perform some examples to numerically verify the theoretical results and demonstrate the performance of the proposed schemes. It should be noted that ETD2 scheme is used for all the following examples while ETD1 scheme is only used in the convergence test due to its low numerical accuracy. The periodic boundary condition is always imposed and the computational domain is set to be $\Omega = (-0.5,0.5)^2$ in two dimensions. 

\subsection{Convergence test}
In this subsection, we test the convergence order of the proposed scheme with $\epsilon=1$ and $\rho^f=0$. Given the initial value $p_0(x,y)=\cos^2(\pi (x+y)), n_0(x,y)=\cos^2(\pi (x-y))$ and the final time $T=0.01$, we begin by testing the temporal convergence order by fixing the spatial mesh size $h_e=1/256$. Since there is no exact solution available, we treat the solution produced by the ETD2 scheme with $\tau_e=T/1024$ as the reference solution. Let $p_{\tau,h}(T)$ be the numerical solution at time $t=T$ for the given time step size $\tau$ and mesh size $h$, and the error function is denoted by $e^p_{\tau,h}=p_{\tau,h}(T)-p_{\tau_e,h_e}(T)$. The temporal errors and the corresponding convergence rates with the successive time step sizes of the ETD1 and ETD2 schemes are reported in Table \ref{tab-temporal-conv}, where the expected temporal convergence accuracies are clearly observed (1 for the ETD1 scheme and 2 for the ETD2 schemes).
\begin{table}[!ht]
	\begin{center}
		\caption{The temporal errors with the corresponding convergence rates of the ETD1 and ETD2 schemes with the fixed mesh size $h=h_e$.}\vspace{0.2cm}
			\begin{tabular}{c|c|cccccc}
				\hline
				&$T/\tau$& $\|e^p_{\tau,h}\|_\infty$ & Rate & $\|e^n_{\tau,h}\|_\infty$ & Rate & $\|e^\phi_{\tau,h}\|_\infty$ & Rate\\
				\hline
\multirow{4}{*}{ETD1 scheme}
&4  & 3.2200e-04 & ---  & 3.2200e-04 & ---  & 6.0699e-06 & ---\\
&8  & 1.5451e-04 & 1.06 & 1.5451e-04 & 1.06 & 2.9209e-06 & 1.06\\
&16 & 7.5254e-05 & 1.04 & 7.5254e-05 & 1.04 & 1.4245e-06 & 1.04\\
&32 & 3.6694e-05 & 1.04 & 3.6694e-05 & 1.04 & 6.9509e-07 & 1.04\\
&64 & 1.7675e-05 & 1.05 & 1.7675e-05 & 1.05 & 3.3492e-07 & 1.05\\
&128& 8.2296e-06 & 1.10 & 8.2296e-05 & 1.10 & 1.5597e-07 & 1.10\\
&256& 3.5230e-06 & 1.22 & 3.5230e-06 & 1.22 & 6.6774e-08 & 1.22\\
&512& 1.1737e-06 & 1.59 & 1.1737e-06 & 1.59 & 2.2246e-08 & 1.59\\
			        \hline
\multirow{4}{*}{ETD2 scheme}
&4  & 2.5425e-05 & ---  & 2.5425e-05 & ---  & 3.5476e-07 & ---\\
&8  & 6.3509e-06 & 2.00 & 6.3509e-06 & 1.99 & 8.8568e-08 & 1.98\\
&16 & 1.5871e-06 & 2.00 & 1.5871e-06 & 2.00 & 2.2131e-08 & 2.00\\
&32 & 3.9647e-07 & 2.00 & 3.9647e-07 & 2.00 & 5.5281e-09 & 2.00\\
&64 & 9.8825e-08 & 2.00 & 9.8825e-07 & 2.00 & 1.3780e-09 & 2.00\\
&128& 2.4416e-08 & 2.02 & 2.4416e-08 & 2.02 & 3.4043e-10 & 2.02\\
&256& 5.8132e-09 & 2.07 & 5.8132e-09 & 2.07 & 8.1056e-11 & 2.07\\
&512& 1.1626e-09 & 2.32 & 1.1626e-09 & 2.32 & 1.6211e-11 & 2.32\\
\hline
		\end{tabular}
		\label{tab-temporal-conv}
	\end{center}
\end{table}

Next, we test the convergence order in space by fixing the time step size $\tau_e=T$. The solution produced by the ETD1 scheme with $h_e=1/1024$ is treated as the reference solution. The spatial errors and the corresponding convergence rates with the successive space step sizes are presented in Table \ref{tab-space-test}. As we can see, the spatial convergence rates are of second-order, which is in line with the theoretical predictions.
\begin{table}[!ht]
	\begin{center}
		\caption{
The spatial errors with the corresponding convergence rates of the ETD1 scheme with the fixed time step size $\tau=\tau_e$.}\vspace{0.2cm}
		{
			\begin{tabular}{c|cccccc}
				\hline
				$1/h$ & $\|e^p_{\tau,h}\|_\infty$ & Rate & $\|e^n_{\tau,h}\|_\infty$ & Rate & $\|e^\phi_{\tau,h}\|_\infty$ & Rate\\
				\hline
8  & 1.8348e-02 & ---  & 1.8348e-02 & ---  & 1.1920e-03 & ---\\
16 & 6.2483e-03 & 1.55 & 6.2483e-03 & 1.55 & 3.6285e-04 & 1.72\\
32 & 1.6885e-03 & 1.89 & 1.6885e-03 & 1.89 & 9.6120e-05 & 1.92\\
64 & 4.3040e-04 & 1.97 & 4.3040e-04 & 1.97 & 2.4391e-05 & 1.98\\
128& 1.0812e-04 & 1.99 & 1.0812e-04 & 1.99 & 6.1208e-06 & 1.99\\
256& 2.7063e-05 & 2.00 & 2.7063e-05 & 2.00 & 1.5316e-06 & 2.00\\
512& 6.7679e-06 & 2.00 & 6.7679e-06 & 2.00 & 3.8300e-07 & 2.00\\
				\hline
		\end{tabular}}
		\label{tab-space-test}
	\end{center}
\end{table}

\subsection{Properties test}
In this subsection, we will simulate the PNP equations to numerical verify the structure preservation of the proposed schemes. It is noted that the ETD2 scheme is used for the following examples. The fixed external charge distribution is given by
\begin{align}
\rho^f(x,y)=200\sum\limits_{\varepsilon_x,\varepsilon_y=\pm 1}\varepsilon_x\varepsilon_y e^{-100[(x+\varepsilon_x x_0)^2+(x+\varepsilon_y y_0)^2]}
\end{align}
with the center $(x_0,y_0)=(0.25,0.25)$. The initial data for concentrations is uniformly set as $p_0=0.1$ and $n_0=0.1$ with the mesh size $h = 1/256$. The dielectric coefficient is set as $\epsilon=1$. Figure \ref{cont_0_phy} presents the evolutions of the minimum value, mass increment and the energy of the numerical solutions produced by ETD2 schemes with $\tau=0.01$ and $\tau=0.001$. The positivity-preservation, mass conservation and energy stability are indeed well preserved. The numerical solution by ETD2 scheme with $\tau=0.001$ at $t=$ 0.003, 0.005, 0.01 and 0.03 is plotted in Figure \ref{cont_0_phase}. We clearly observe that as time evolves, the mobile ions are attracted by opposite fixed charges. 
\begin{figure}[!ht]
	\centerline{
		{\includegraphics[width=0.34\textwidth]{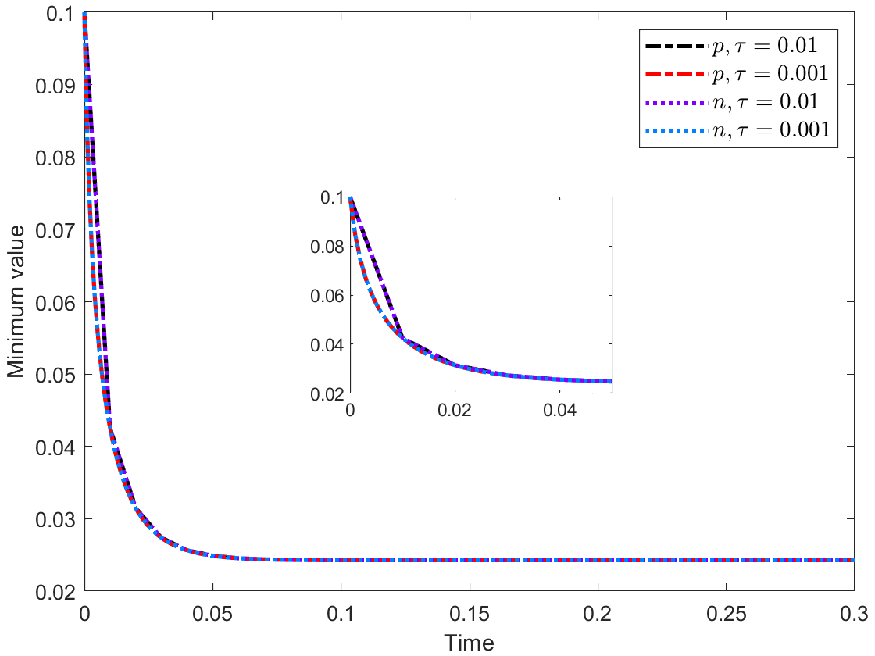}}\hspace{-0.1cm}
		{\includegraphics[width=0.34\textwidth]{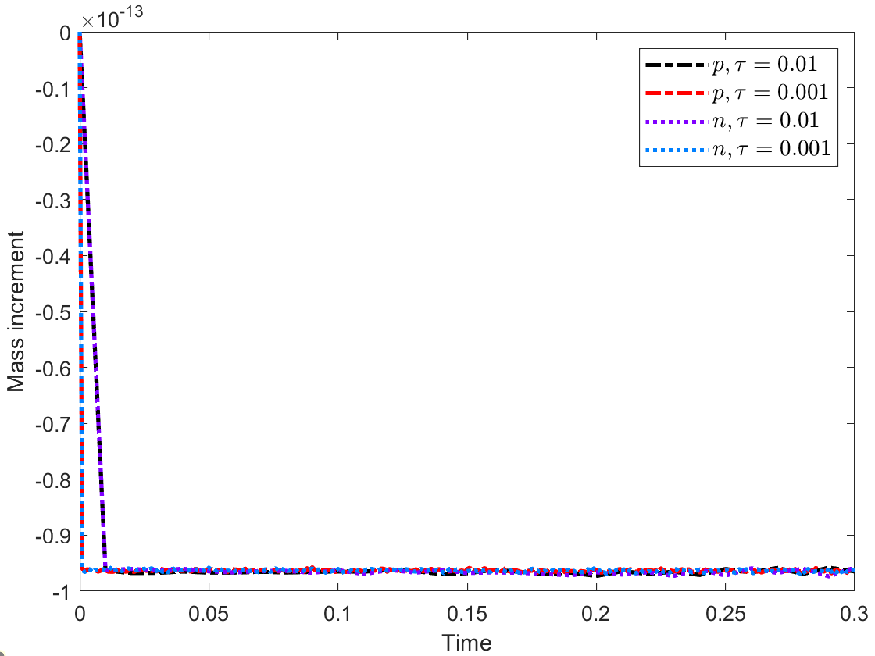}}\hspace{-0.1cm}
		{\includegraphics[width=0.34\textwidth]{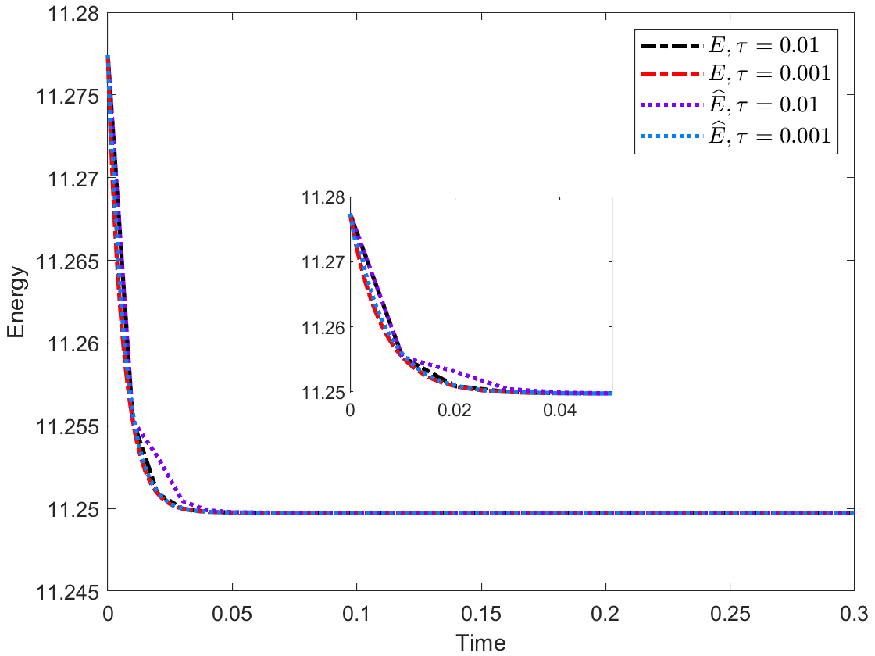}}}
	\caption{Evolutions of the minimum value, mass increment and the energy of numerical solution produced by the ETD2 scheme for PNP equations with $\tau=0.01$ and $\tau=0.001$.}
	\label{cont_0_phy}
\end{figure}
\begin{figure}[!ht]
	\centerline{
		{\includegraphics[width=0.26\textwidth]{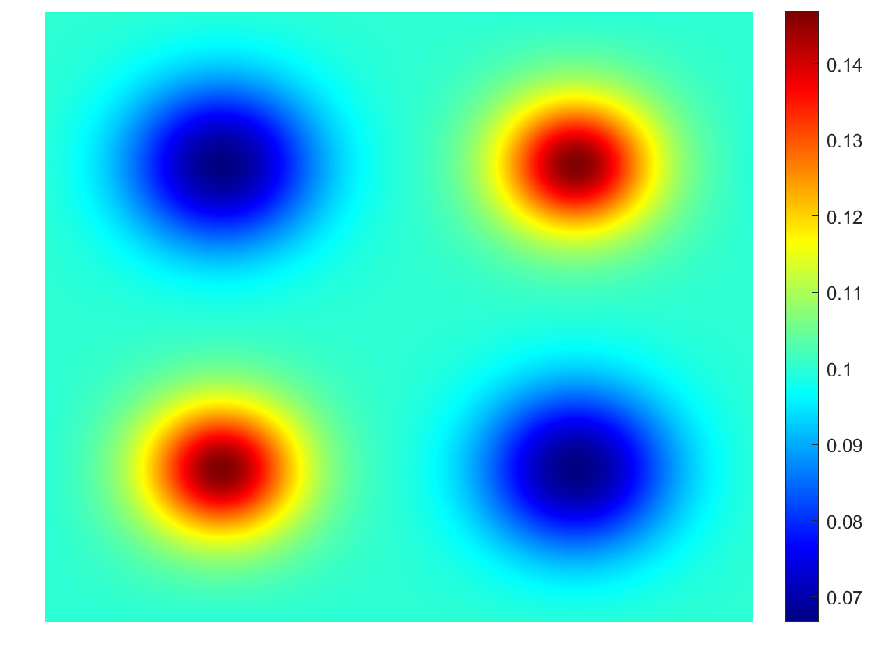}}\hspace{-0.1cm}
		{\includegraphics[width=0.26\textwidth]{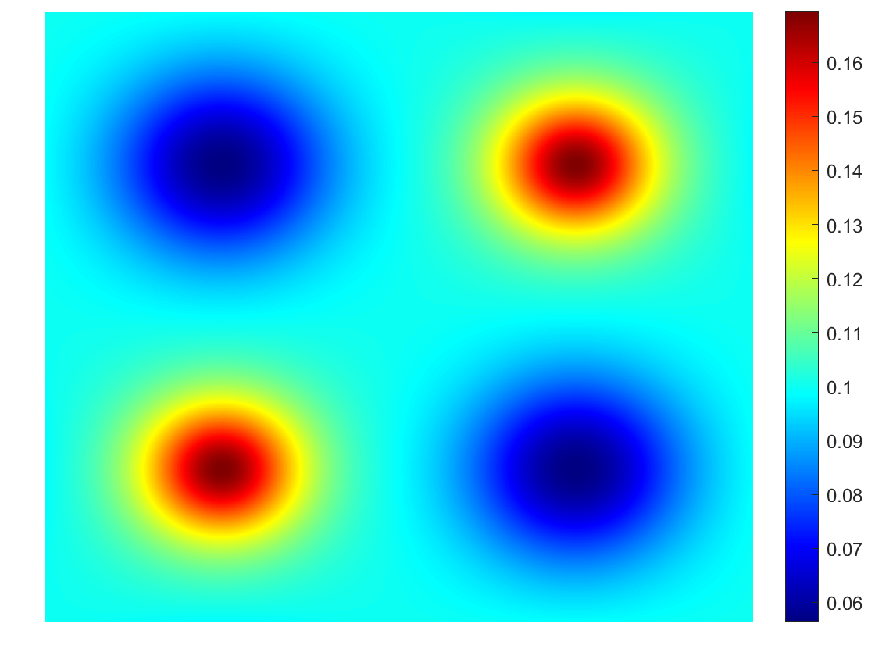}}\hspace{-0.1cm}		    {\includegraphics[width=0.26\textwidth]{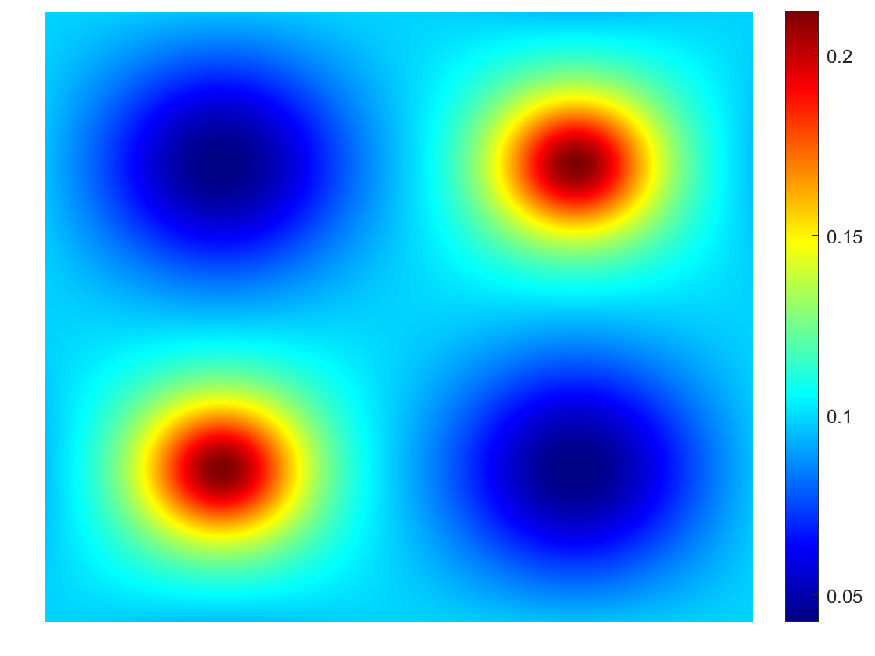}}\hspace{-0.1cm}
		{\includegraphics[width=0.26\textwidth]{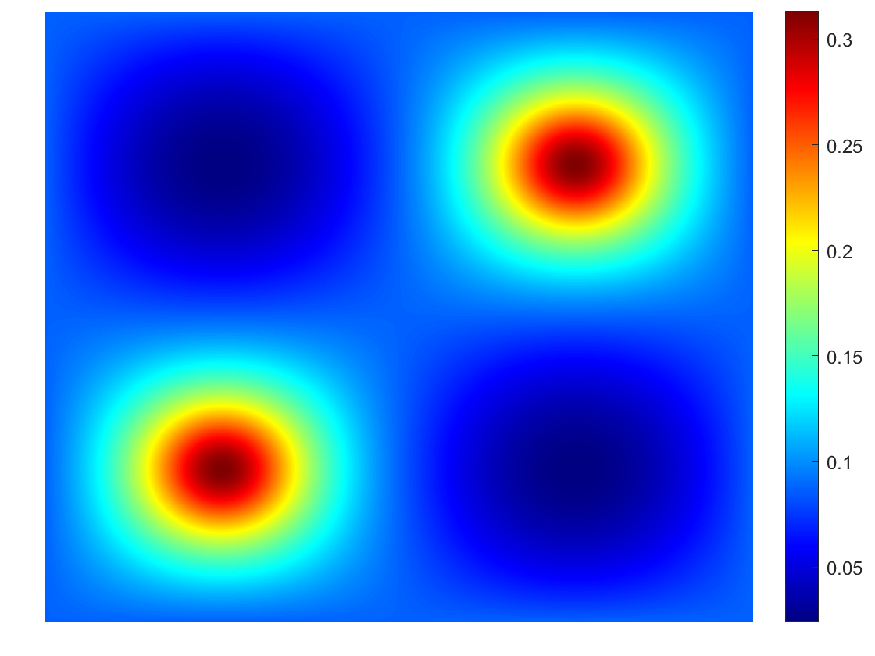}}}
		\centerline{
		{\includegraphics[width=0.26\textwidth]{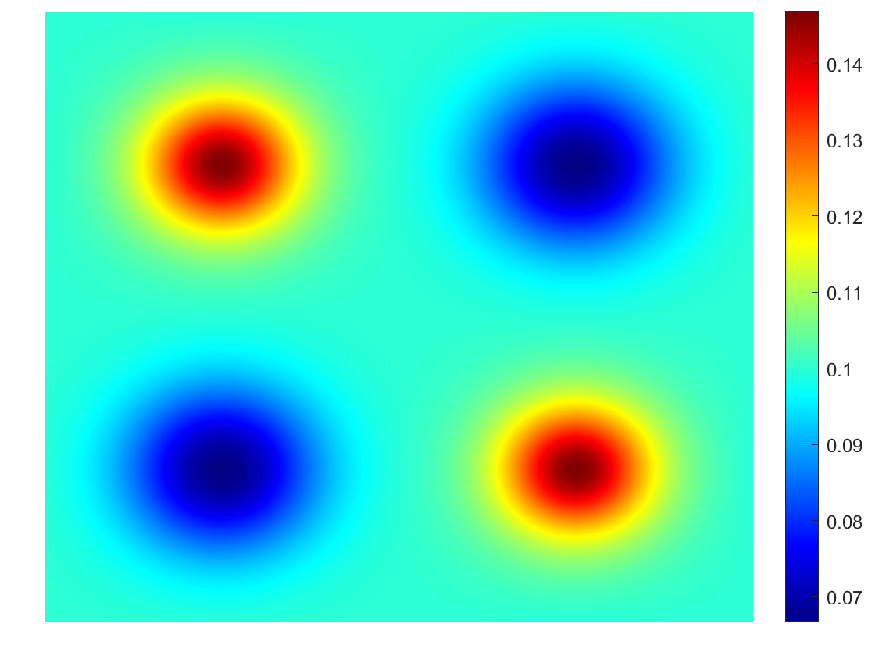}}\hspace{-0.1cm}
		{\includegraphics[width=0.26\textwidth]{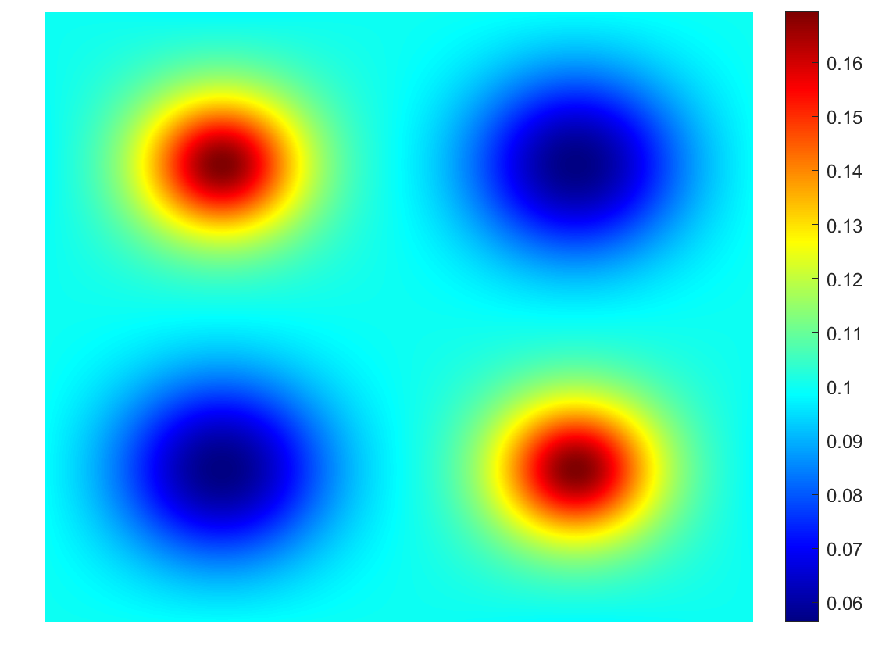}}\hspace{-0.1cm}		    {\includegraphics[width=0.26\textwidth]{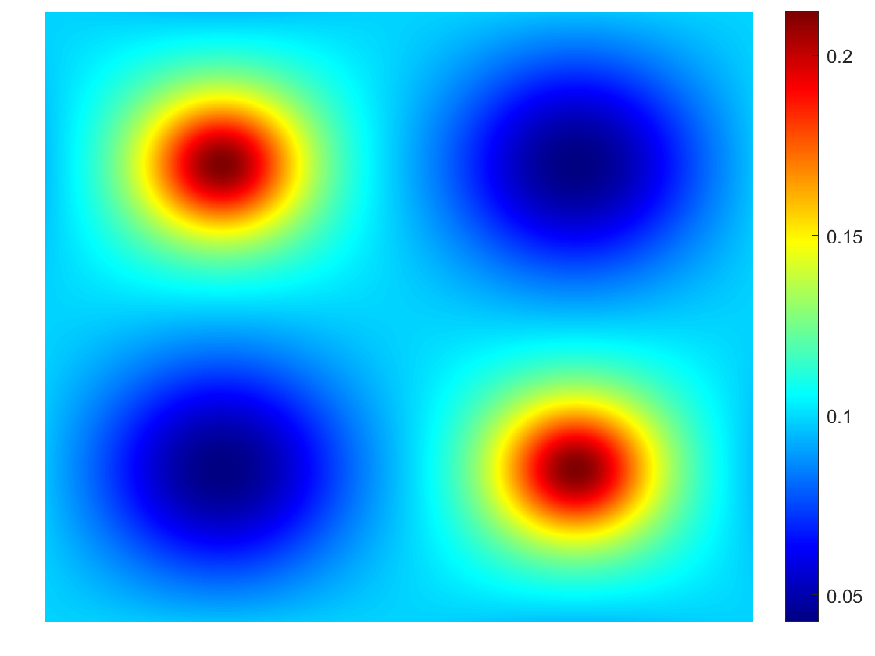}}\hspace{-0.1cm}
		{\includegraphics[width=0.26\textwidth]{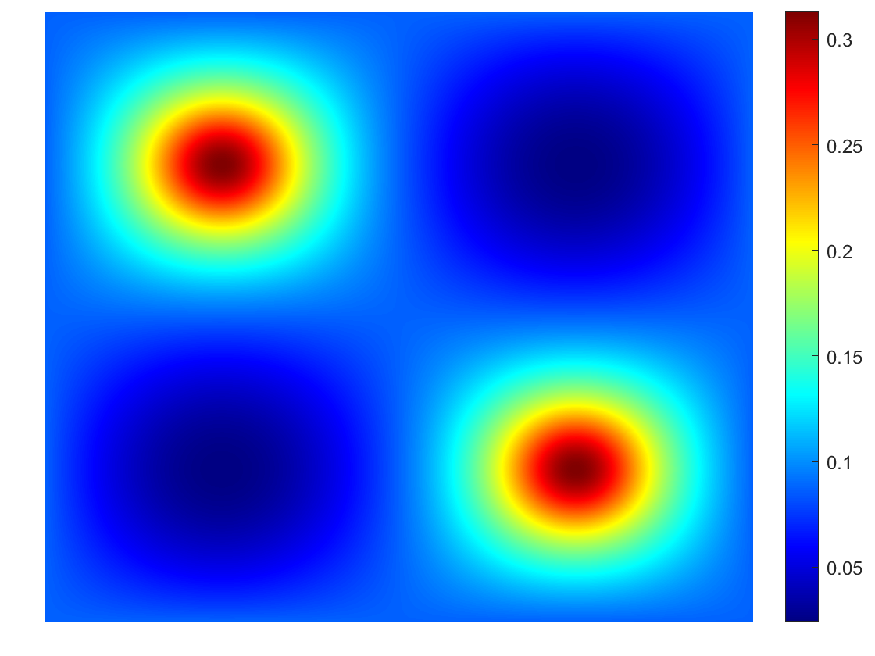}}}
	\caption{Profiles of concentrations produced by the ETD2 scheme at $t$ = 0.003, 0.005, 0.01 and 0.3 (left to right) for the PNP equations with $\tau=0.001$.Top: positive icon $p$, bottom, negative icon $n$.}
	\label{cont_0_phase}
\end{figure}
\subsection{Discontinuity in the steady state solution}
In this subsection, we will simulate the PNP equations with the discontinuous initial value. The fixed external charge distribution is given by
\begin{align*}
\rho^f(x,y)=4\chi_{[0.15,0.25]\times[0.15,0.25]}.
\end{align*}
with the dielectric coefficient $\epsilon=1$. The initial data for concentrations is uniformly set as $p_0(x,y)=\chi_{[0,0.2]\times[0,0.2]}$ and $n_0(x,y)=2\chi_{[0,0.2]\times[0,0.2]}$. The time step size is chosen as $\tau=$ 0.01 and the mesh size $h = 1/256$. Figure \ref{disc_1_phase} illustrates the configurations of the numerical solutions at $t$ = 0.02, 0.04, 0.06 and 0.1. The corresponding developments of the supremum norm and the energy are plotted in Figure \ref{disc_1_phy}. We can observe that the positivity-preservation, mass conservation and energy stability are numerically preserved, which is in agreement with the theoretical results.
\begin{figure}[!ht]
	\centerline{
		{\includegraphics[width=0.26\textwidth]{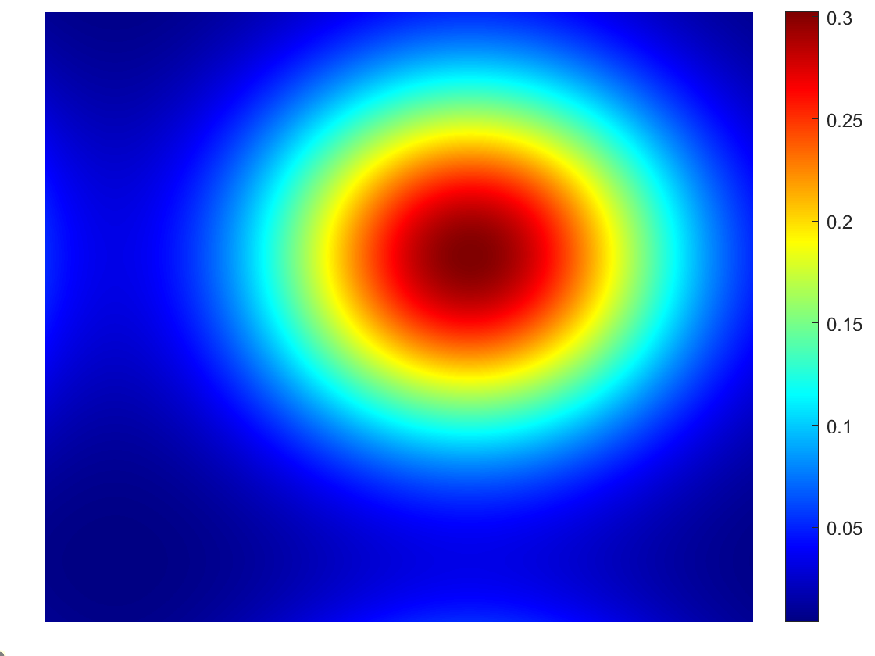}}\hspace{-0.1cm}
		{\includegraphics[width=0.26\textwidth]{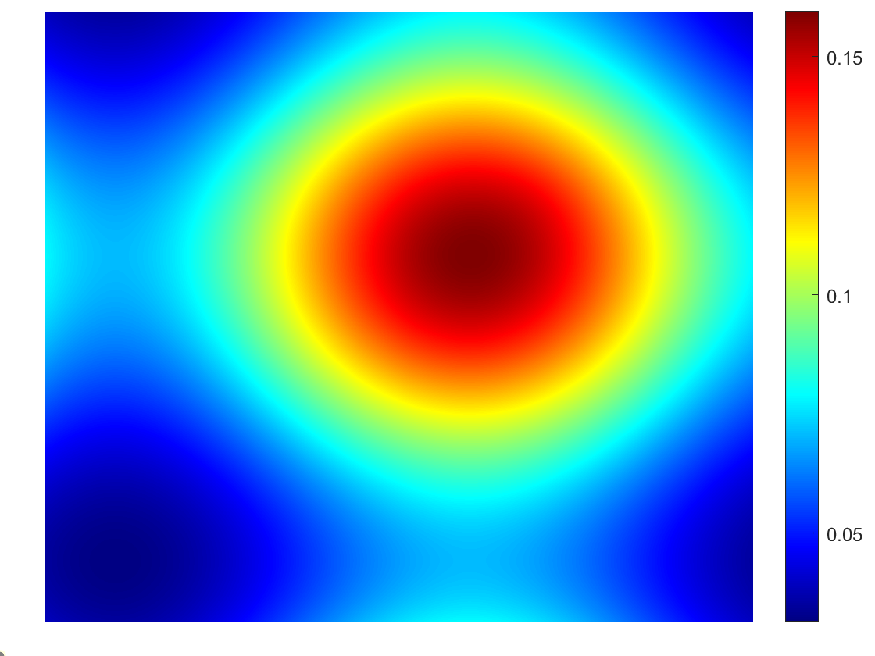}}\hspace{-0.1cm}		    {\includegraphics[width=0.26\textwidth]{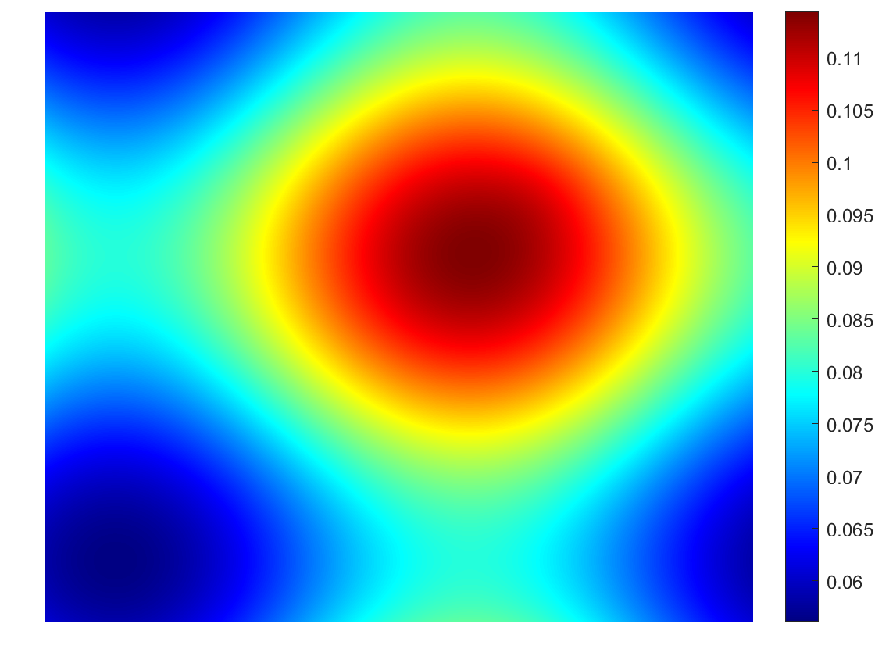}}\hspace{-0.1cm}
		{\includegraphics[width=0.26\textwidth]{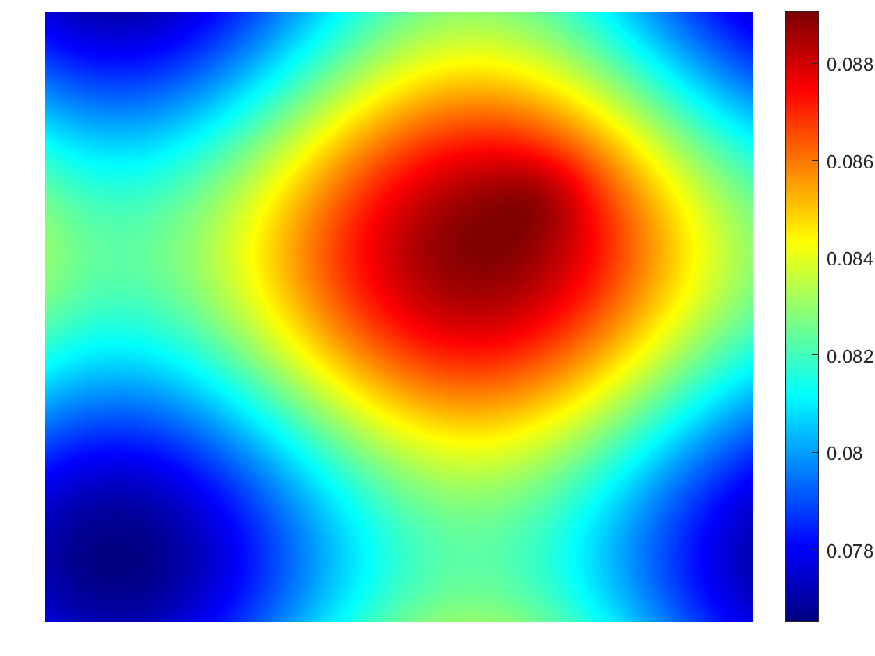}}}
		\centerline{
		{\includegraphics[width=0.26\textwidth]{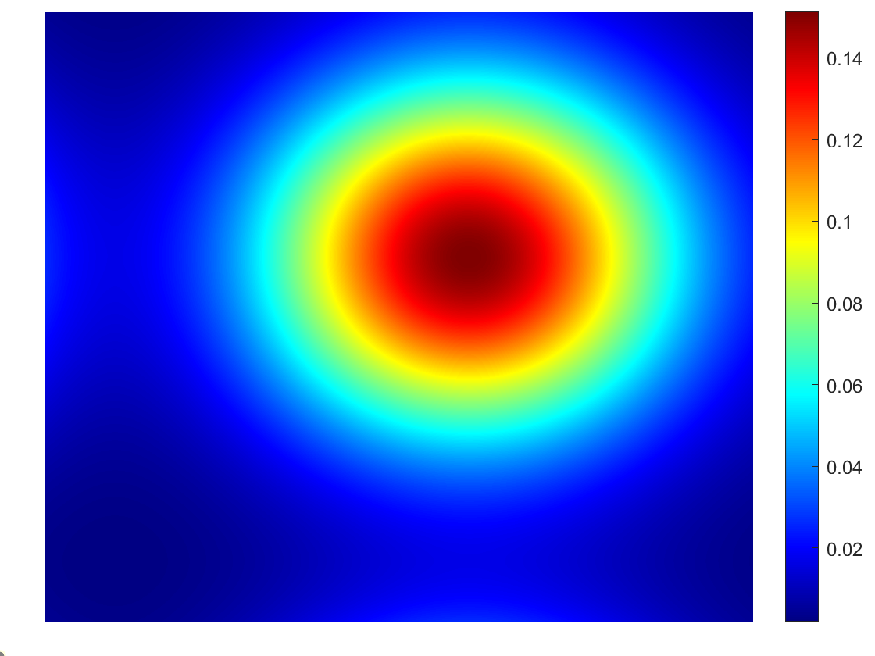}}\hspace{-0.1cm}
		{\includegraphics[width=0.26\textwidth]{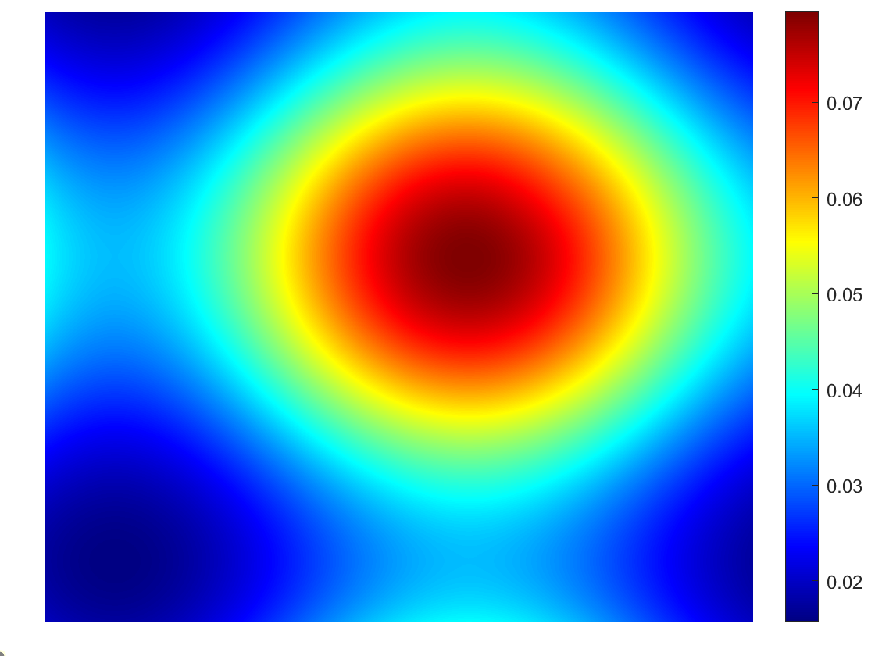}}\hspace{-0.1cm}		    {\includegraphics[width=0.26\textwidth]{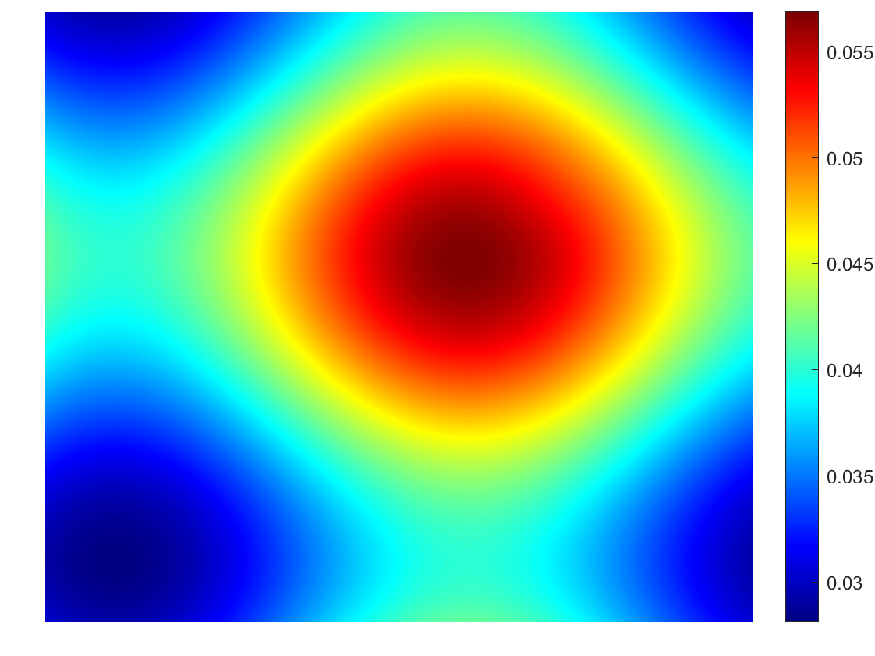}}\hspace{-0.1cm}
		{\includegraphics[width=0.26\textwidth]{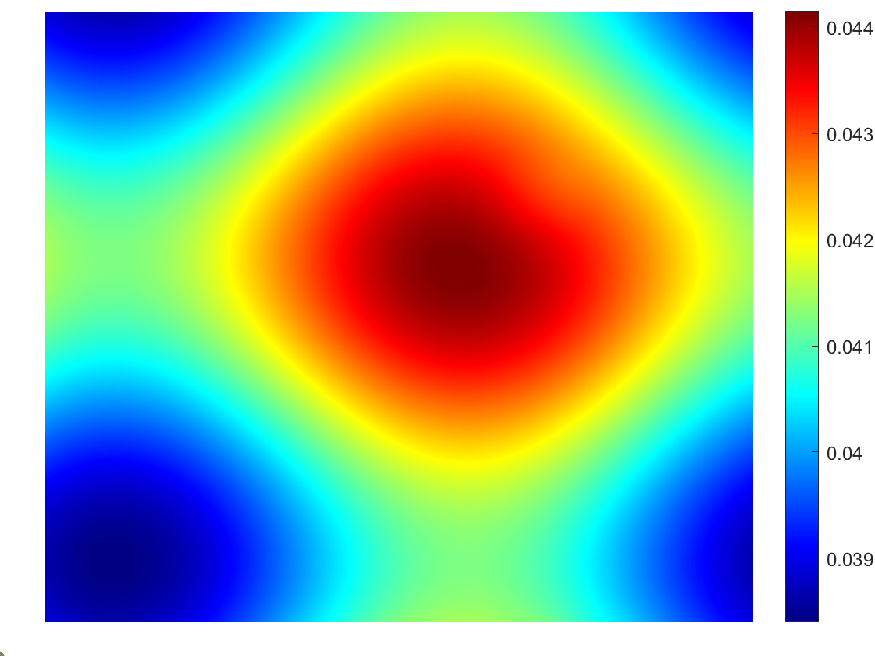}}}
	\caption{Profiles of concentrations produced by the ETD2 scheme with $\tau=0.01$ at $t$ = 0.02, 0.04, 0.06 and 0.1 (left to right) for the PNP equations with $\epsilon=1$. Top: positive icon $p$, bottom, negative icon $n$.}
	\label{disc_1_phase}
\end{figure}
\begin{figure}[!ht]
	\centerline{
		{\includegraphics[width=0.34\textwidth]{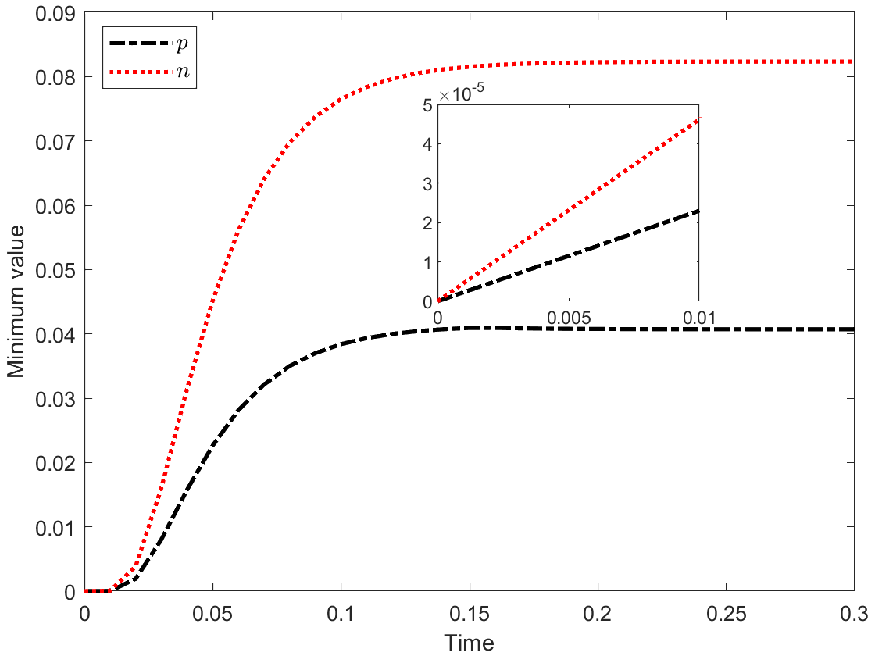}}\hspace{-0.1cm}
		{\includegraphics[width=0.34\textwidth]{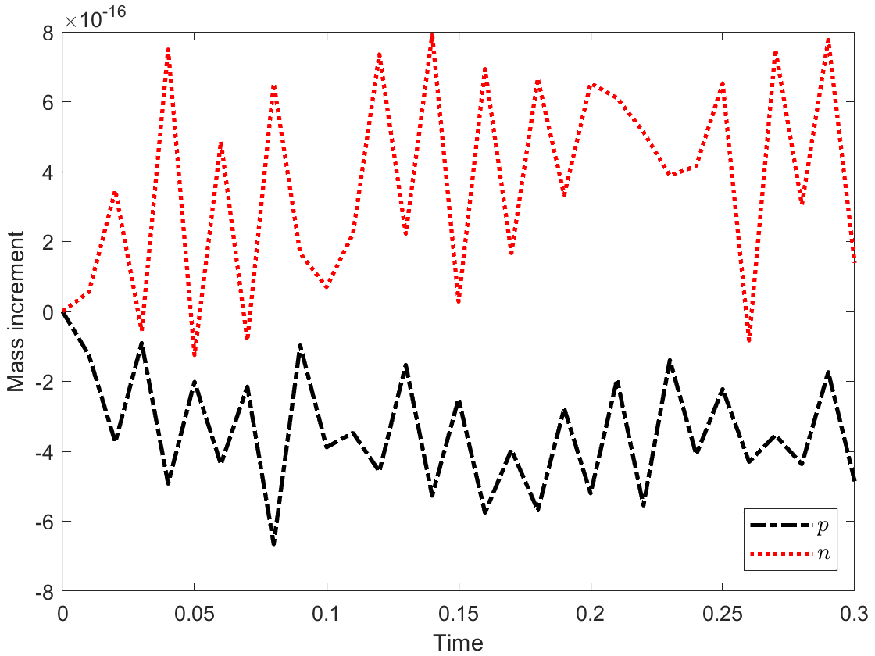}}\hspace{-0.1cm}
		{\includegraphics[width=0.34\textwidth]{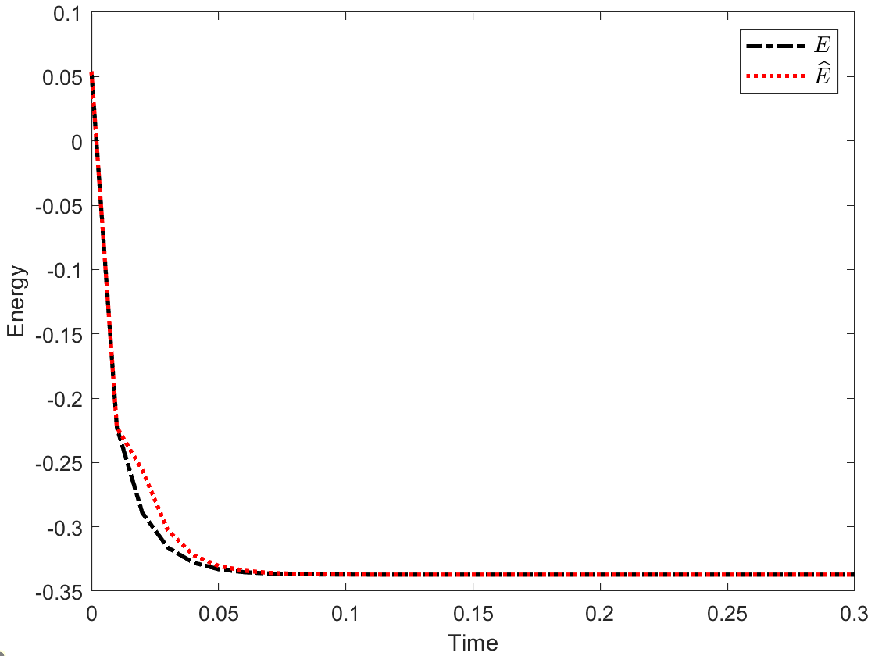}}}
	\caption{Evolutions of the minimum value, mass increment and the energy of numerical solution produced by the ETD2 scheme with $\tau=0.01$ for PNP equations with $\epsilon=1$.}
	\label{disc_1_phy}
\end{figure}

We then consider the dielectric coefficient as $\epsilon=0.1$. With the same parameters, Figure \ref{disc_01_phase} illustrates the configurations of the numerical solutions at $t$ = 0.02, 0.04, 0.06 and 0.1. The corresponding developments of the supremum norm and the energy are plotted in Figure \ref{disc_01_phy}. From these figures, we observe that positive and negative ions gradually accumulate due to lower diffusion of the electrostatic potential than that in the case $\epsilon=1$. Meanwhile, the positivity-preservation, mass conservation and energy stability are well preserved.
\begin{figure}[!ht]
	\centerline{
		{\includegraphics[width=0.26\textwidth]{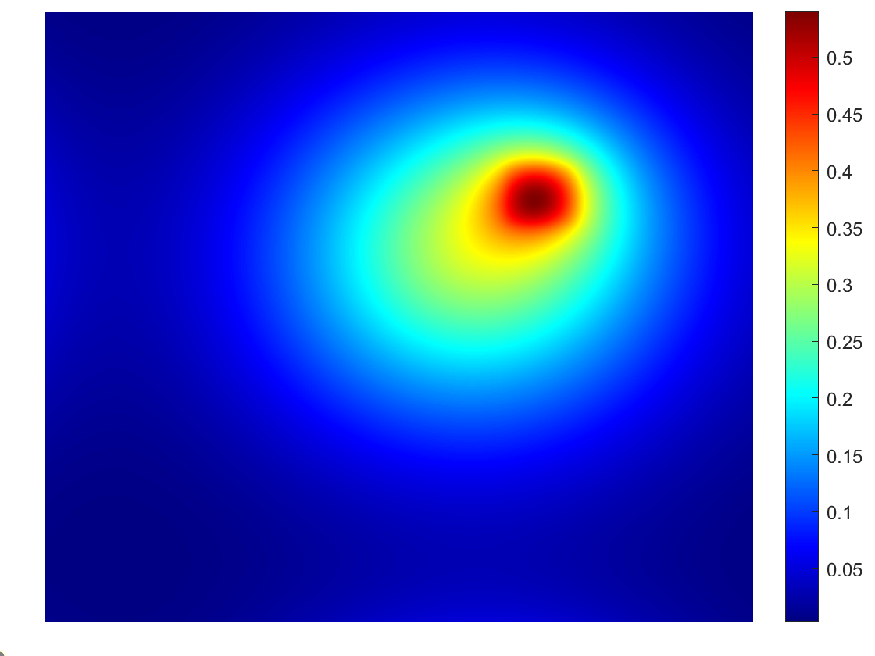}}\hspace{-0.1cm}
		{\includegraphics[width=0.26\textwidth]{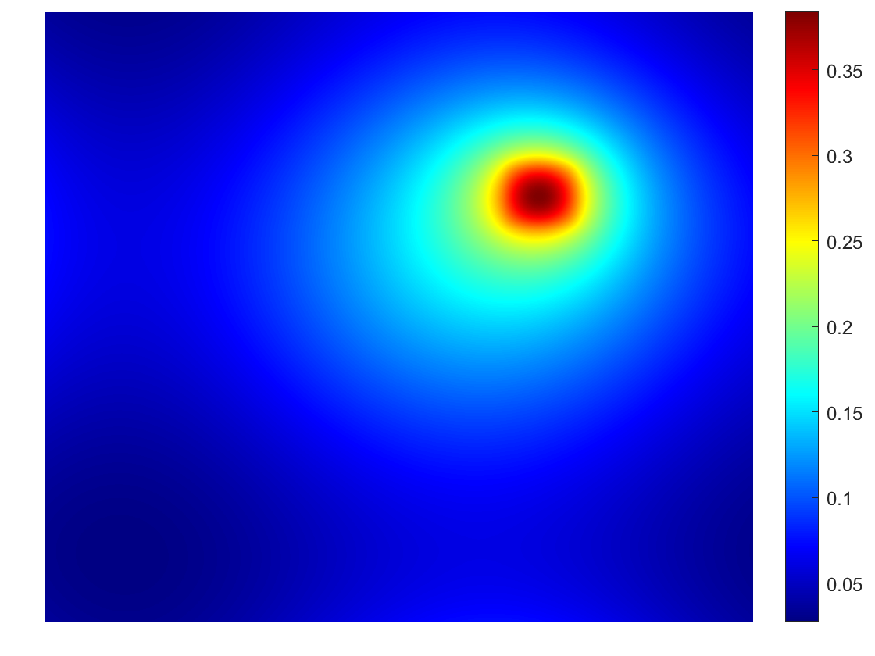}}\hspace{-0.1cm}		    {\includegraphics[width=0.26\textwidth]{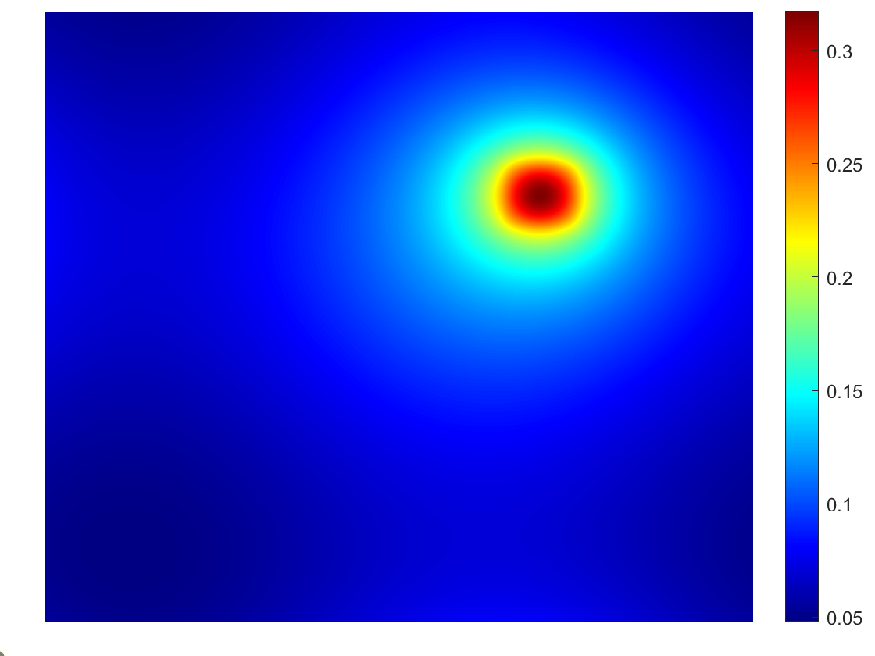}}\hspace{-0.1cm}
		{\includegraphics[width=0.26\textwidth]{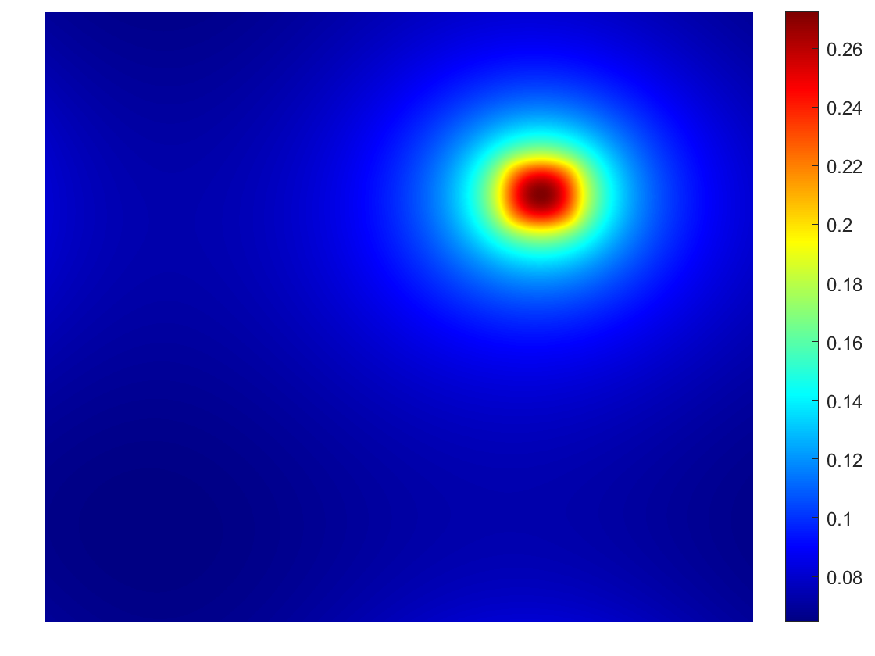}}}
		\centerline{
		{\includegraphics[width=0.26\textwidth]{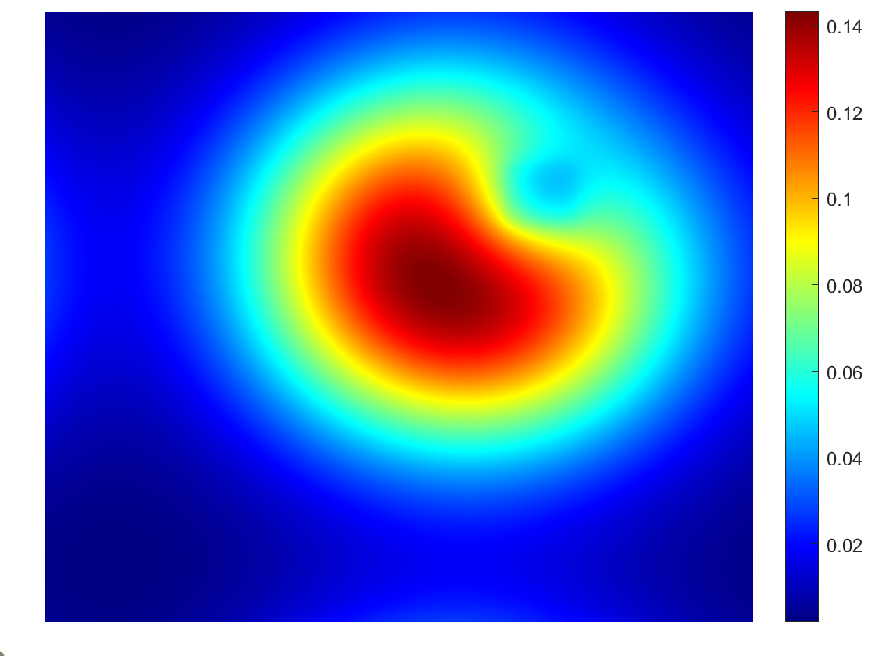}}\hspace{-0.1cm}
		{\includegraphics[width=0.26\textwidth]{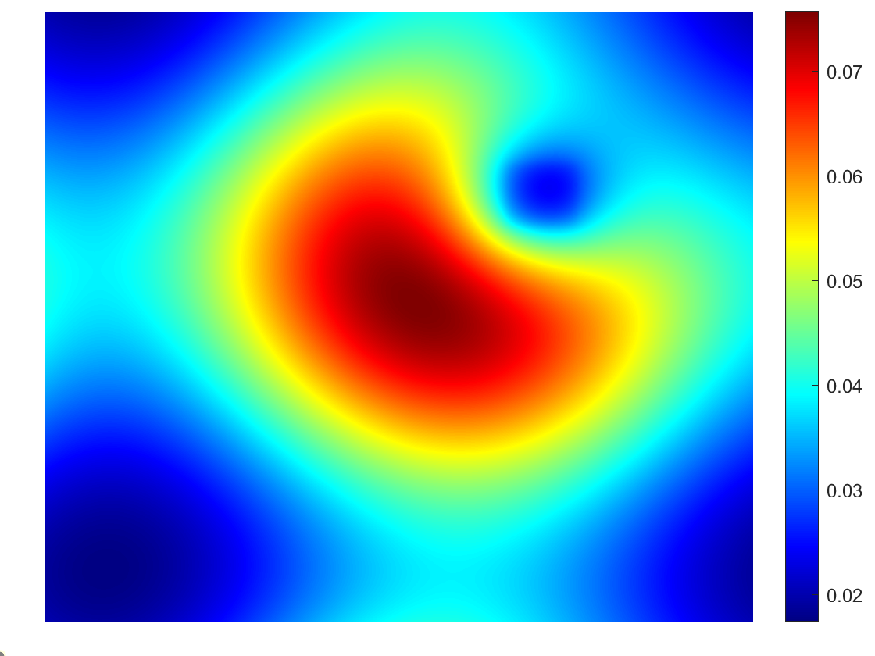}}\hspace{-0.1cm}		    {\includegraphics[width=0.26\textwidth]{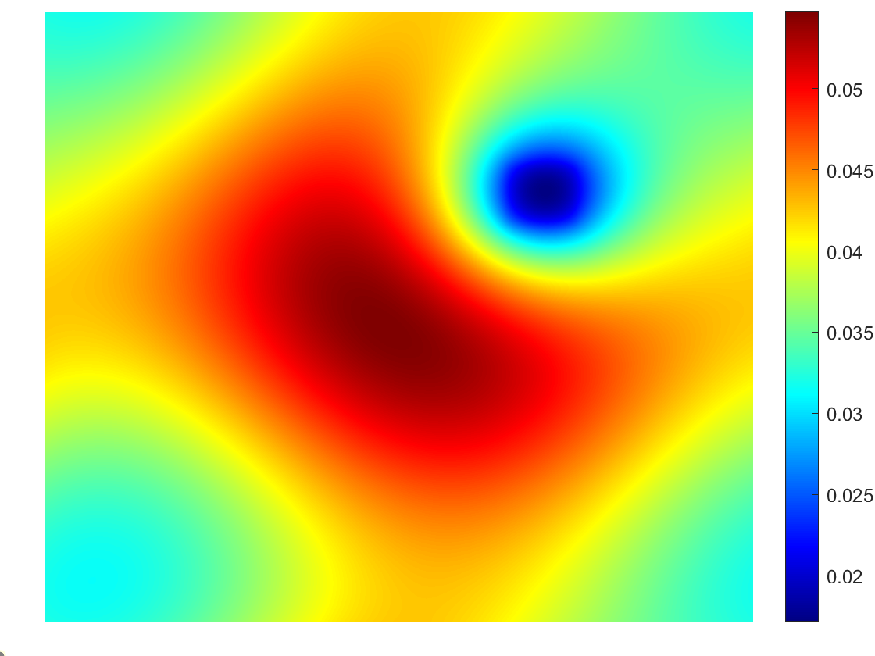}}\hspace{-0.1cm}
		{\includegraphics[width=0.26\textwidth]{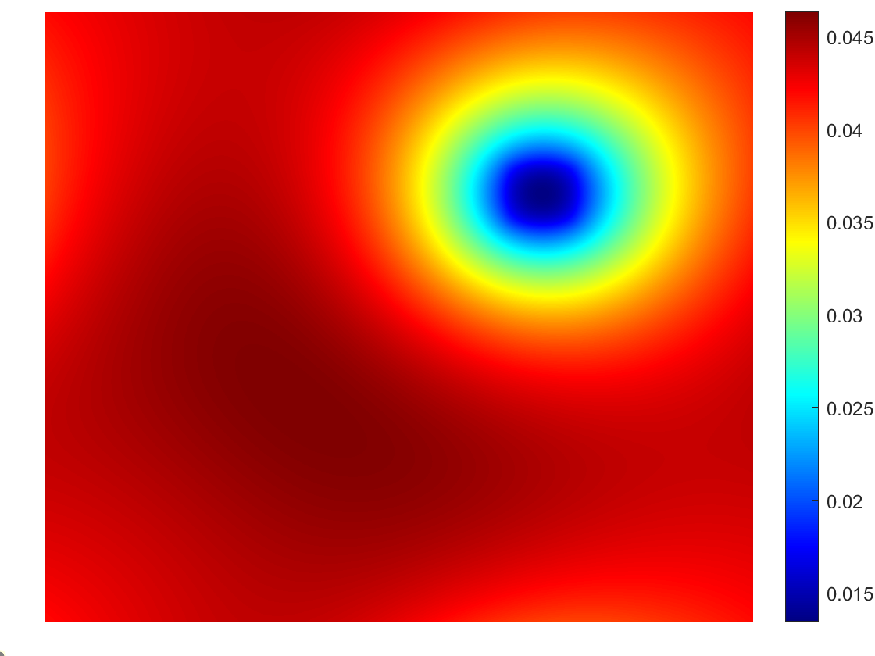}}}
	\caption{Profiles of concentrations produced by the ETD2 scheme with $\tau=0.01$ at $t$ = 0.02, 0.04, 0.06 and 0.1 (left to right) for the PNP equations with $\epsilon=0.1$. Top: positive icon $p$, bottom, negative icon $n$.}
	\label{disc_01_phase}
\end{figure}
\begin{figure}[!ht]
	\centerline{
		{\includegraphics[width=0.34\textwidth]{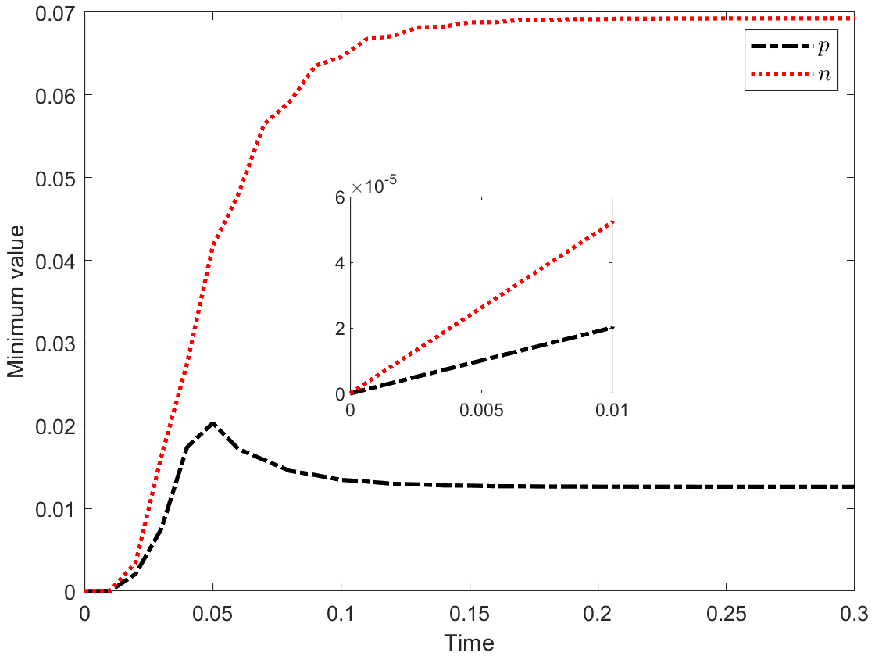}}\hspace{-0.1cm}
		{\includegraphics[width=0.34\textwidth]{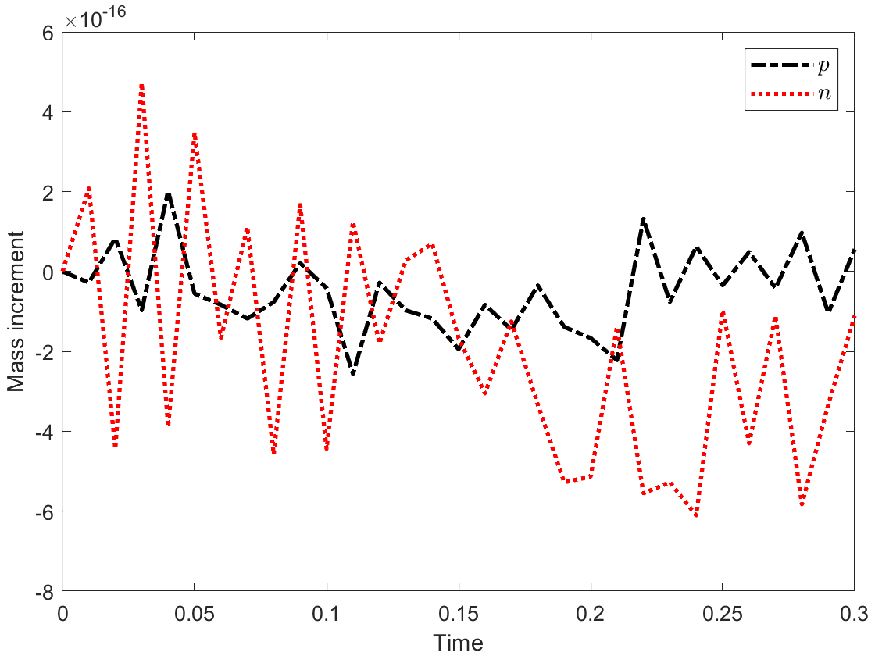}}\hspace{-0.1cm}
		{\includegraphics[width=0.34\textwidth]{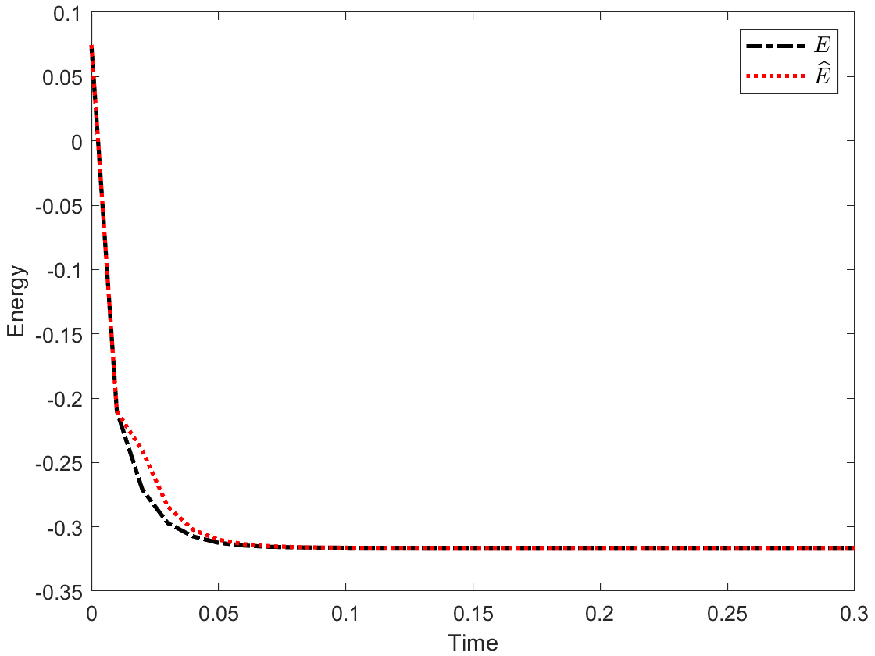}}}
	\caption{Evolutions of the minimum value, mass increment and the energy of numerical solution produced by the ETD2 scheme with $\tau=0.01$ for PNP equations with $\epsilon=0.1$.}
	\label{disc_01_phy}
\end{figure}

\subsection{Simulations of sodium-chloride saline solution}
We finally perform numerical experiments to simulate a sodium-chloride saline solution. The initial configuration for two species is generated by random distribution of a mean concentration of 0.5 on the entire domain. The fixed charge is given by 
\begin{align*}
\rho^f(x,y)=\left\{
\begin{array}{ll}
\rho_0, &x=0.25,\\
-\rho_0,&x=-0.25,\\
0,&\text{otherwise},
\end{array}
\right.
\end{align*}
where $\rho_0>0$ is the charge density. The time-space size is set as $\tau=0.01$ and $h=1/256$. 

We first take the charge density $\rho_0=1$. Figure \ref{saline_1_phase} illustrates the snapshots of numerical solutions produced by ETD2 scheme at $t=$ 0.001, 0.005, 0.02, and 0.05, respectively. The corresponding evolutions of the minimum value, mass conservation and the energy is plotted in Figure \ref{saline_1_phy}. From these figures, we observe that starting from the heterogeneous initial condition, the ionic distribution smooths out rapidly at $t=0.01$ first due to diffusion. The positivity preservation, mass conservation and energy stability are also well preserved.
\begin{figure}[!ht]
	\centerline{
		{\includegraphics[width=0.26\textwidth]{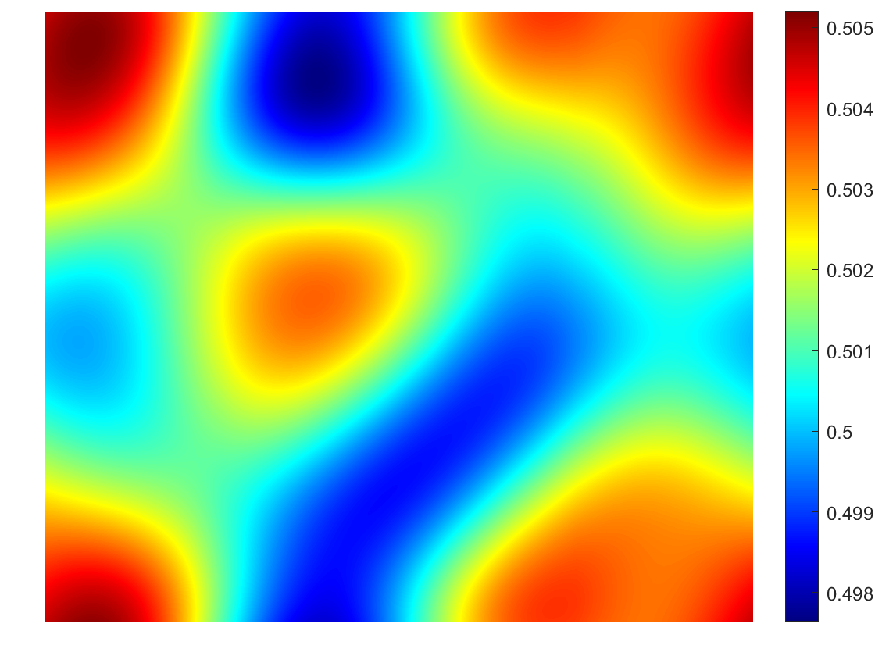}}\hspace{-0.1cm}
		{\includegraphics[width=0.26\textwidth]{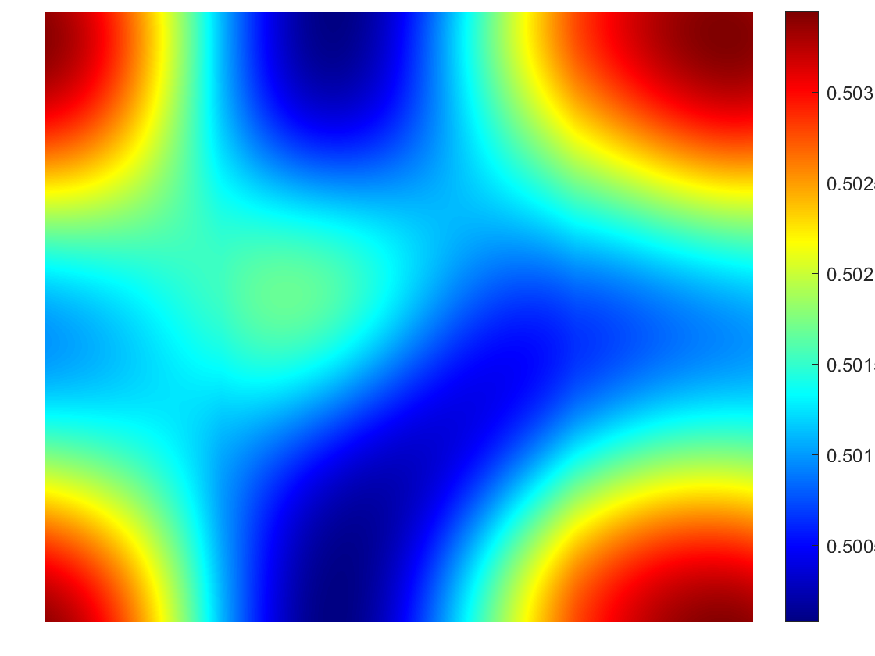}}\hspace{-0.1cm}		    {\includegraphics[width=0.26\textwidth]{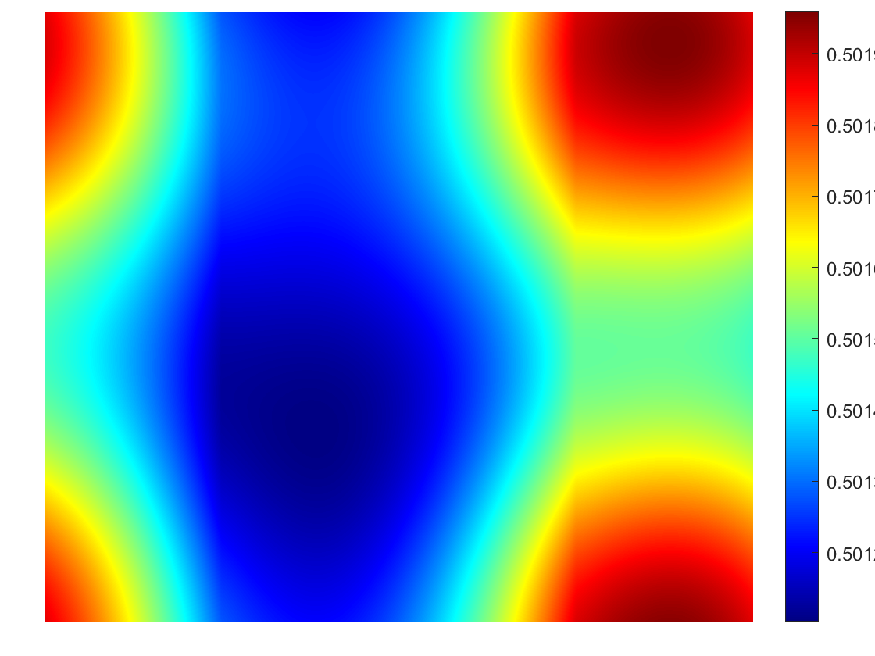}}\hspace{-0.1cm}
		{\includegraphics[width=0.26\textwidth]{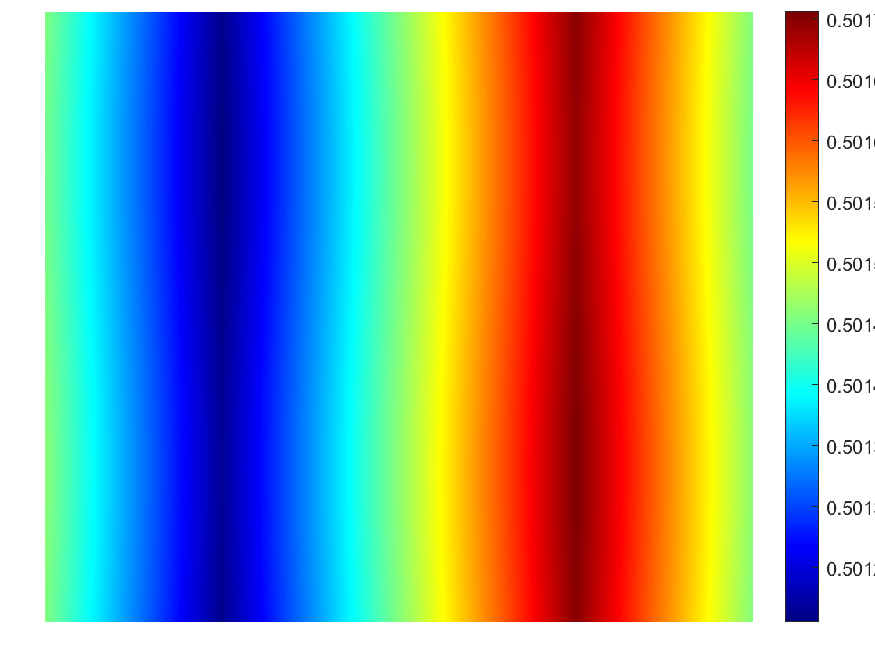}}}
		\centerline{
		{\includegraphics[width=0.26\textwidth]{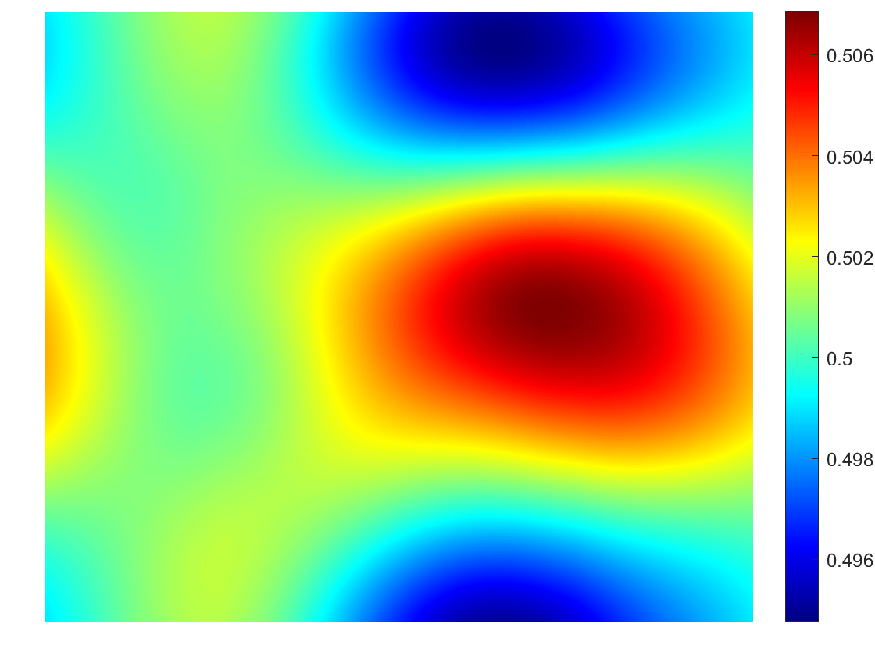}}\hspace{-0.1cm}
		{\includegraphics[width=0.26\textwidth]{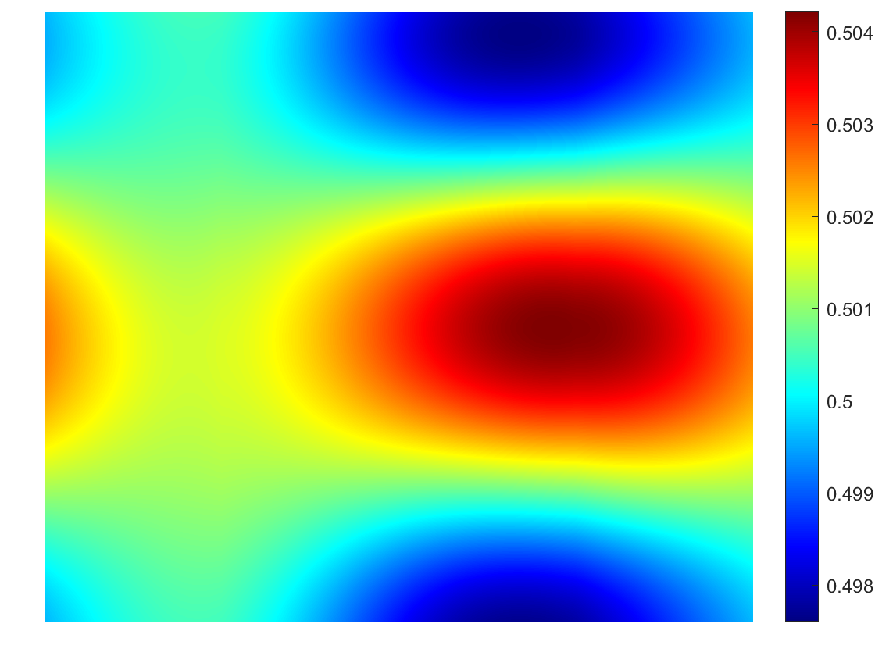}}\hspace{-0.1cm}		    {\includegraphics[width=0.26\textwidth]{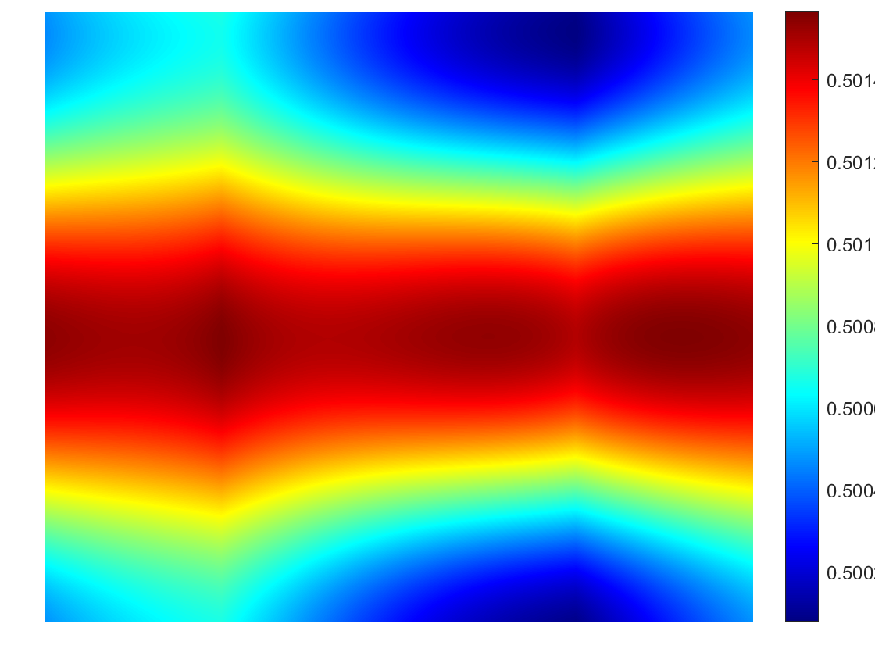}}\hspace{-0.1cm}
		{\includegraphics[width=0.26\textwidth]{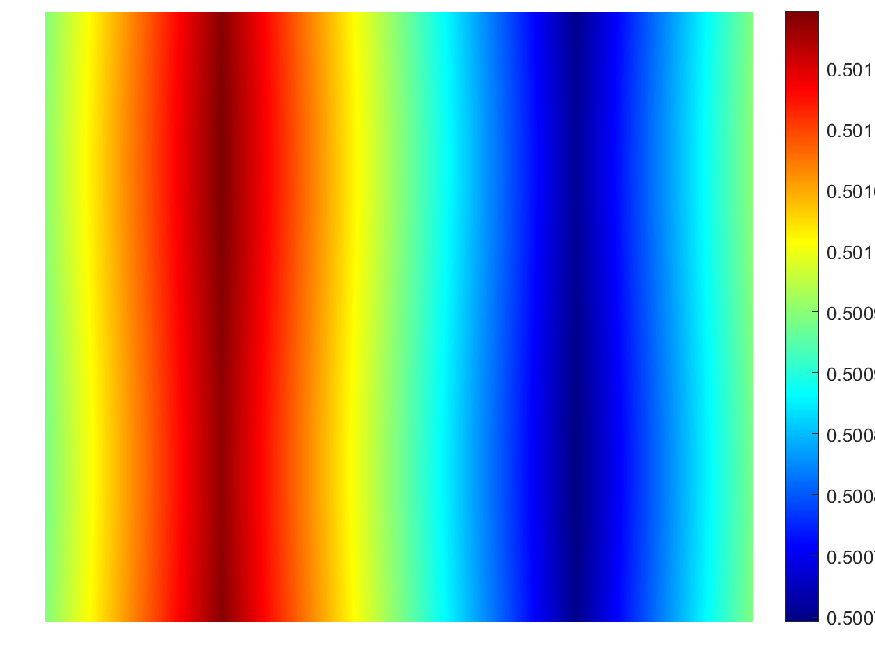}}}
	\caption{Profiles of concentrations produced by the ETD2 scheme with $\tau=0.01$ at $t$ = 0.01, 0.02, 0.05 and 0.3 (left to right) for the PNP equations with $\rho_0=1$. Top: positive icon $p$, bottom, negative icon $n$.}
	\label{saline_1_phase}
\end{figure}
\begin{figure}[!ht]
	\centerline{
		{\includegraphics[width=0.34\textwidth]{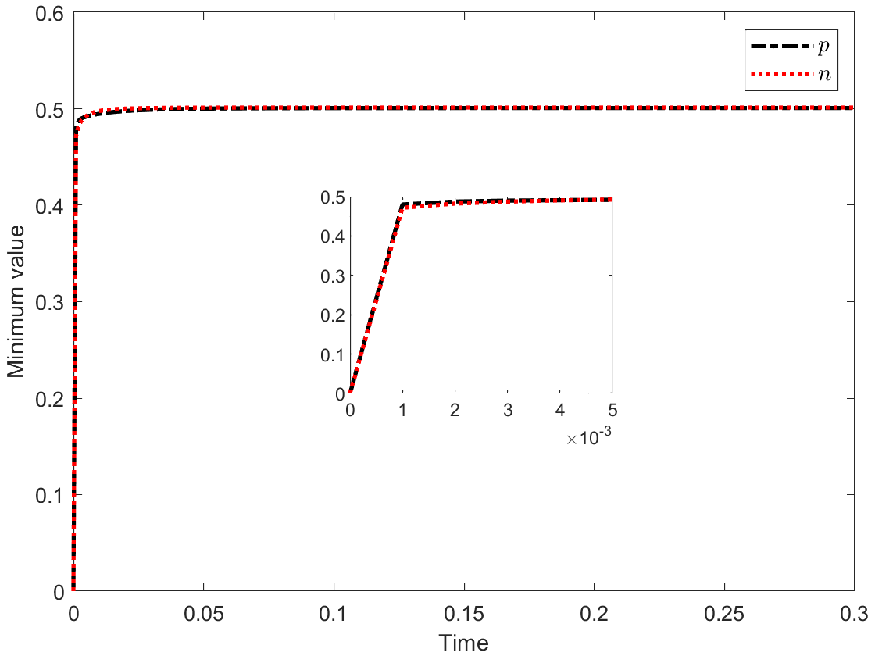}}\hspace{-0.1cm}
		{\includegraphics[width=0.34\textwidth]{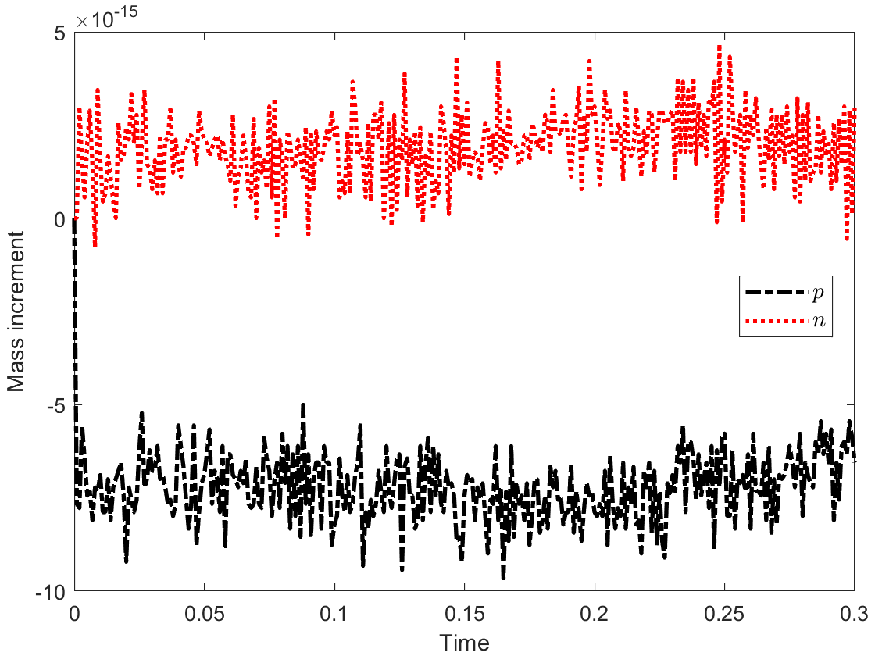}}\hspace{-0.1cm}
		{\includegraphics[width=0.34\textwidth]{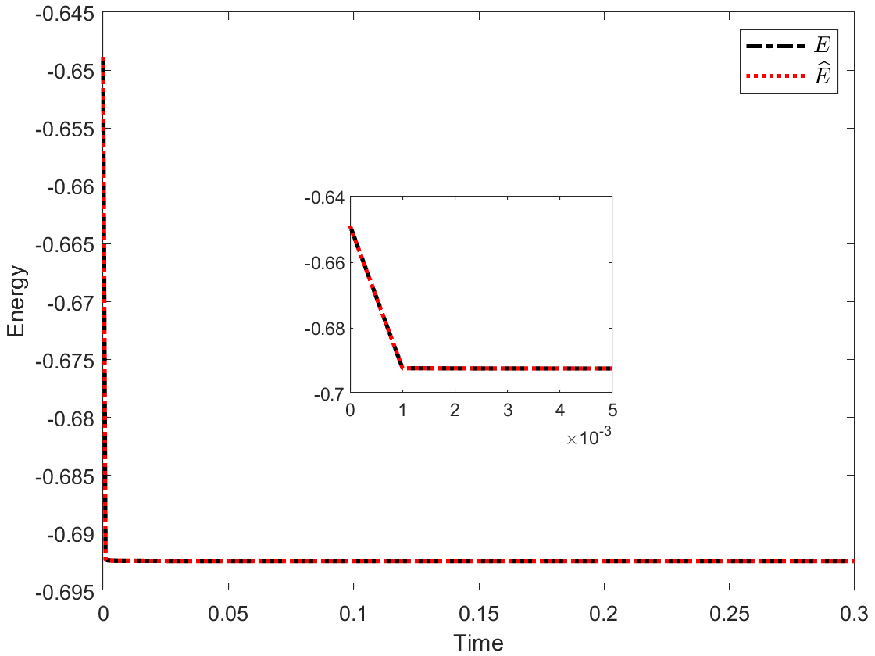}}}
	\caption{Evolutions of the minimum value, mass increment and the energy of numerical solution produced by the ETD2 scheme with $\tau=0.01$ for PNP equations with $\rho_0=1$.}
	\label{saline_1_phy}
\end{figure}

Next we increase the fixed charge density to $\rho_0=10$. Figure \ref{saline_10_phase} illustrates the snapshots of numerical solutions produced by ETD2 scheme at $t=$ 0.001, 0.005, 0.02, and 0.05, respectively. The corresponding evolutions of the minimum value, mass conservation and the energy is plotted in Figure \ref{saline_10_phy}. It is observed that the positivity preservation, mass conservation and energy stability are also well preserved.
\begin{figure}[!ht]
	\centerline{
		{\includegraphics[width=0.26\textwidth]{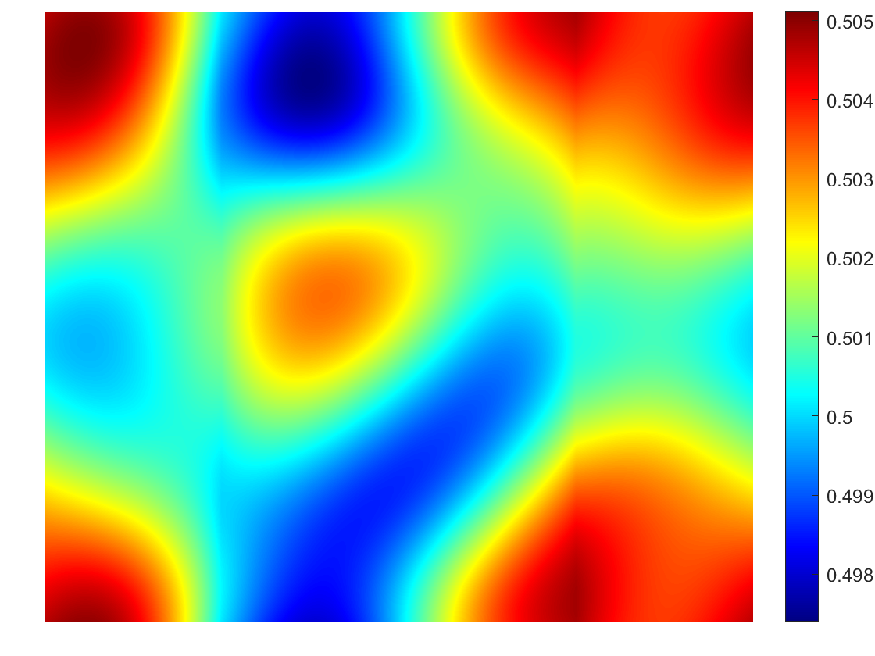}}\hspace{-0.1cm}
		{\includegraphics[width=0.26\textwidth]{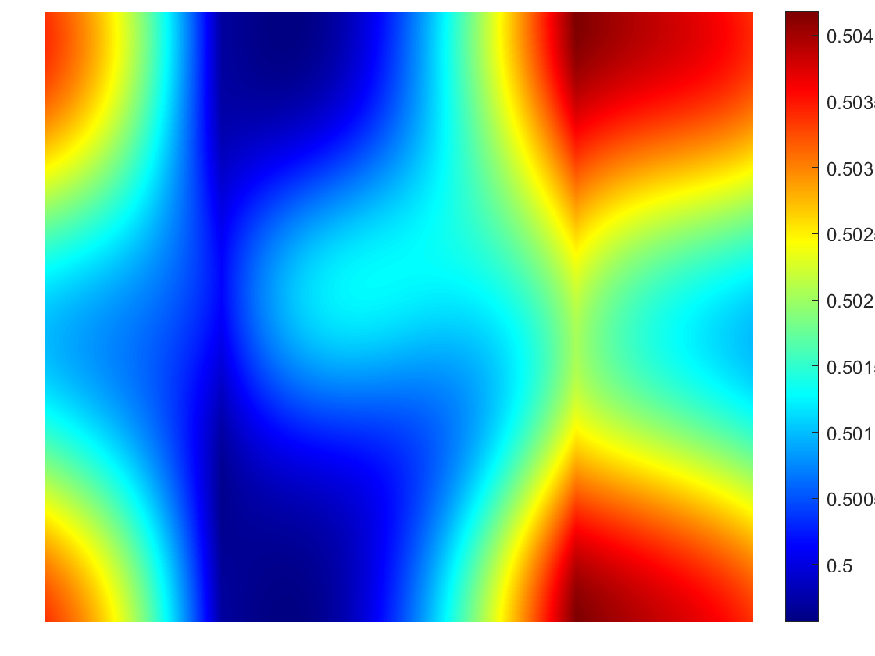}}\hspace{-0.1cm}		    {\includegraphics[width=0.26\textwidth]{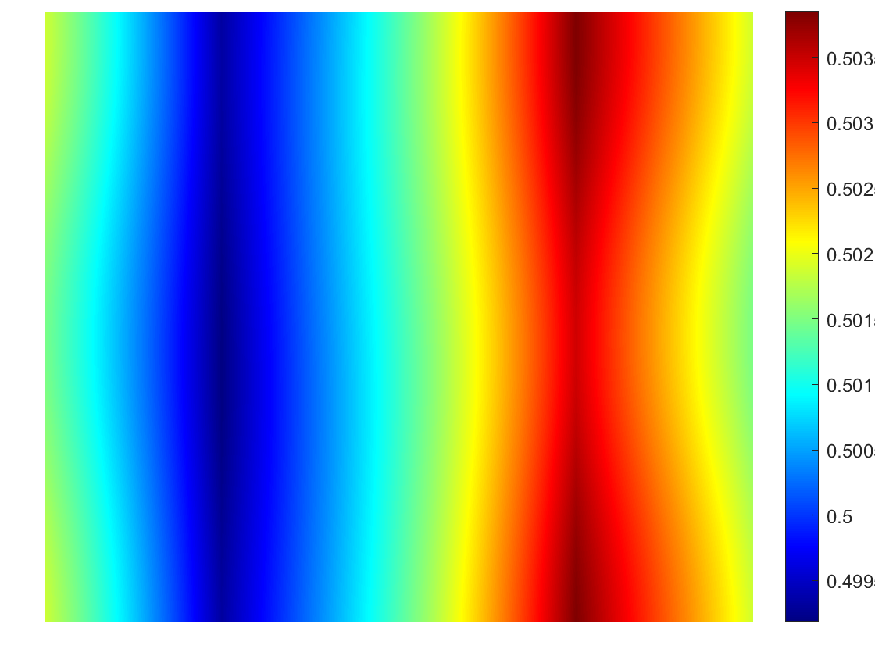}}\hspace{-0.1cm}
		{\includegraphics[width=0.26\textwidth]{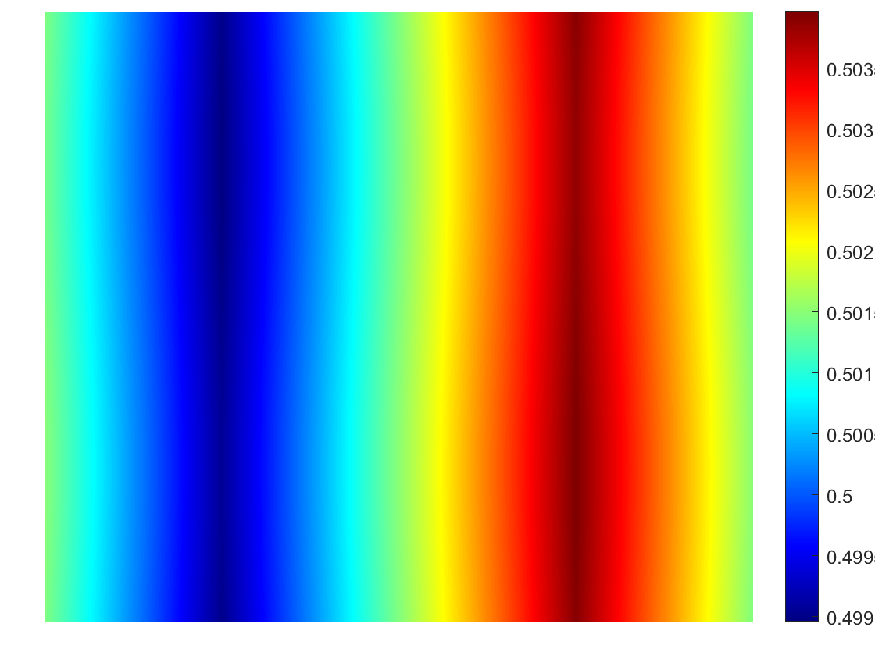}}}
		\centerline{
		{\includegraphics[width=0.26\textwidth]{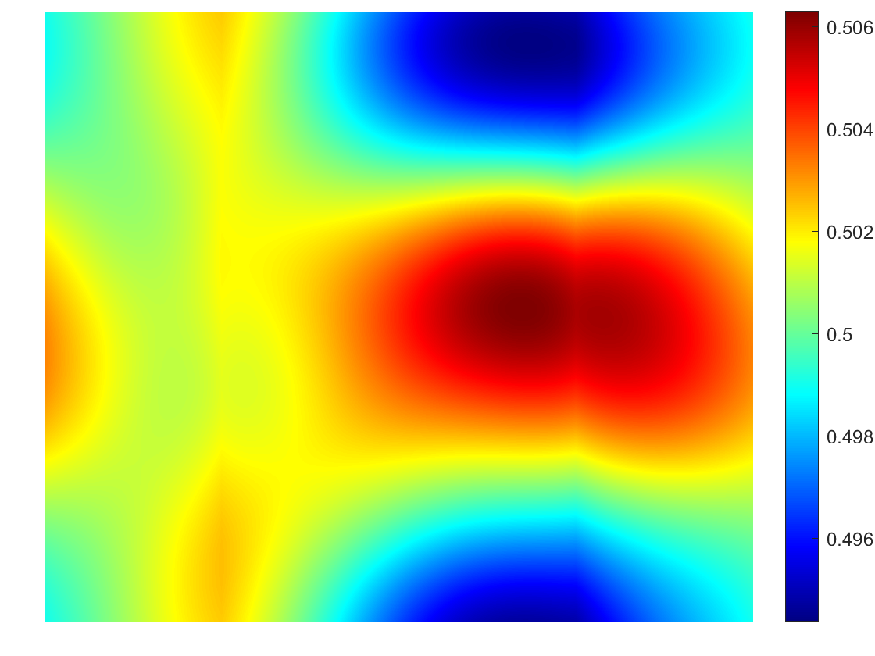}}\hspace{-0.1cm}
		{\includegraphics[width=0.26\textwidth]{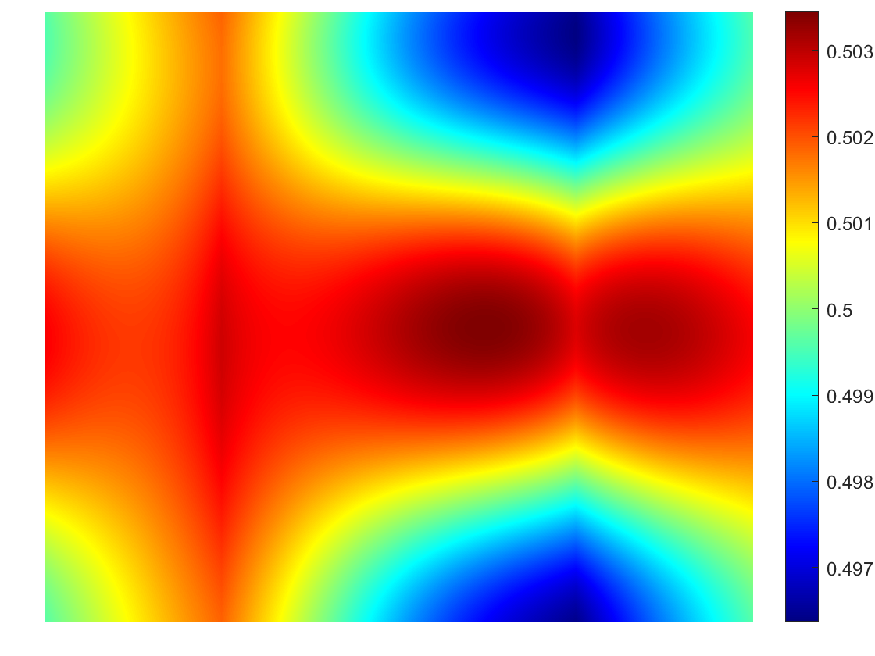}}\hspace{-0.1cm}		    {\includegraphics[width=0.26\textwidth]{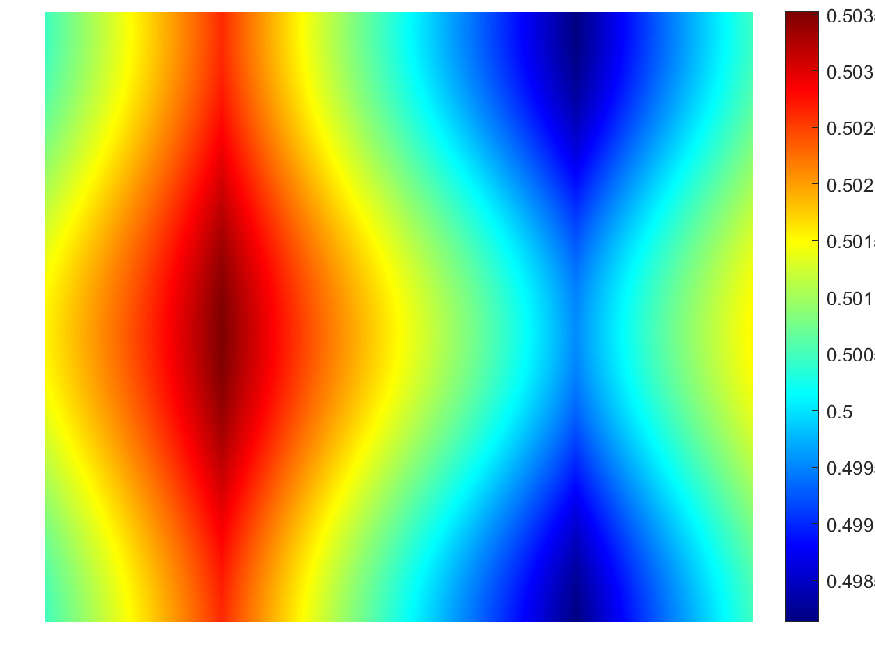}}\hspace{-0.1cm}
		{\includegraphics[width=0.26\textwidth]{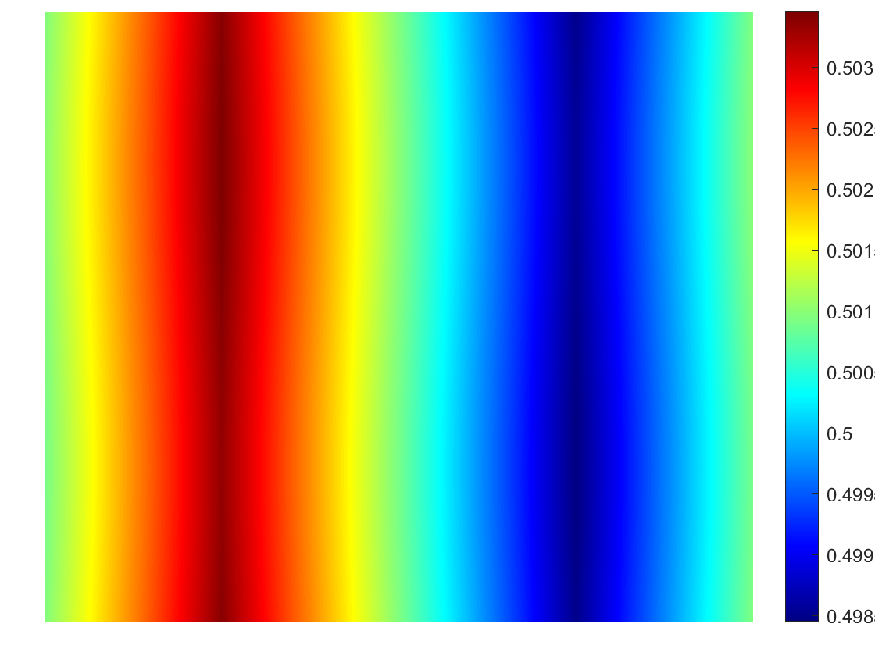}}}
	\caption{Profiles of concentrations produced by the ETD2 scheme with $\tau=0.01$ at $t$ = 0.01, 0.02, 0.05 and 0.3 (left to right) for the PNP equations with $\rho_0=10$. Top: positive icon $p$, bottom, negative icon $n$.}
	\label{saline_10_phase}
\end{figure}
\begin{figure}[!ht]
	\centerline{
		{\includegraphics[width=0.34\textwidth]{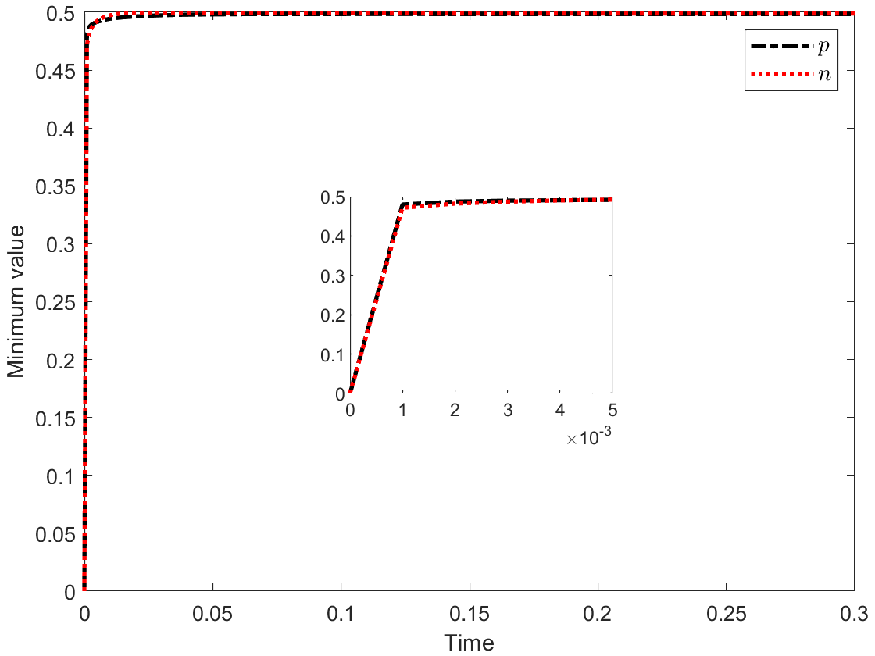}}\hspace{-0.1cm}
		{\includegraphics[width=0.34\textwidth]{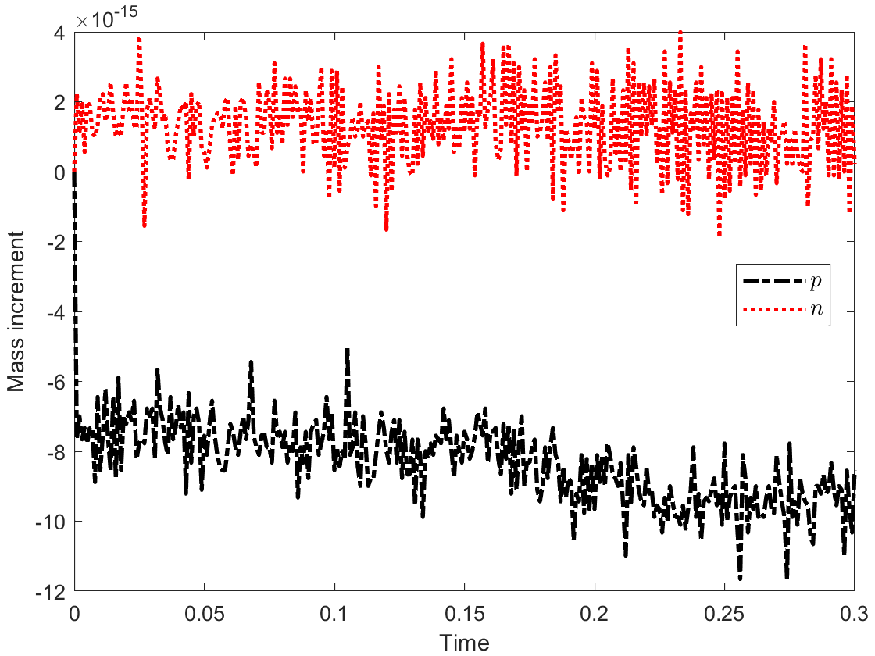}}\hspace{-0.1cm}
		{\includegraphics[width=0.34\textwidth]{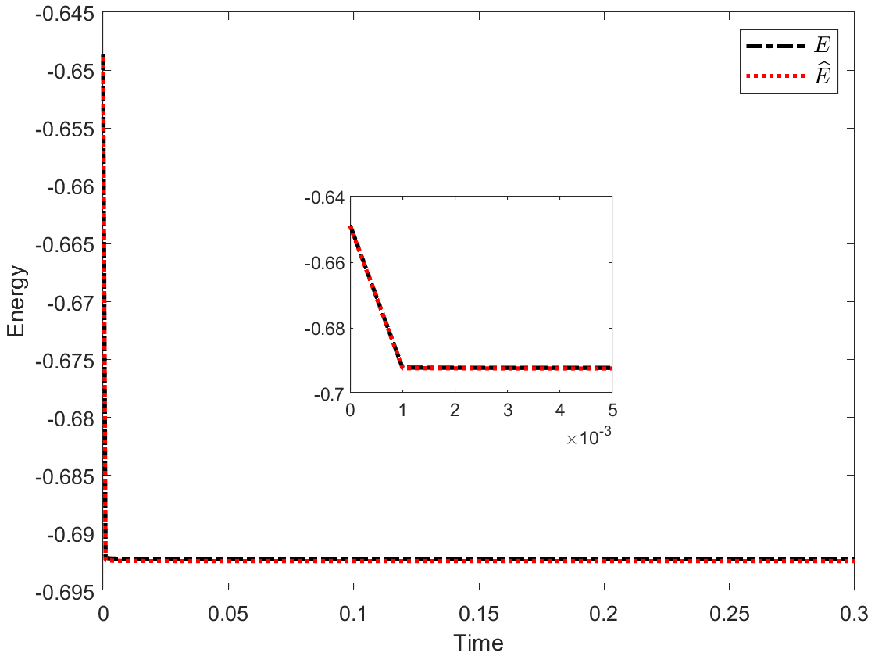}}}
	\caption{Evolutions of the minimum value, mass increment and the energy of numerical solution produced by the ETD2 scheme with $\tau=0.01$ for PNP equations with $\rho_0=10$.}
	\label{saline_10_phy}
\end{figure}

Then we numerically test extreme cases by increasing the fixed charge density to $\rho_0=50$. As shown in Figure \ref{saline_50_phase}, counterions of positive icon in the steady state are strongly attracted to the fixed charge due to electrostatic interactions, giving rise to a peak at $x=0.25$ and extremely low counterion concentration in the rest of the region. Figure \ref{saline_50_phy} illustrates the evolutions of the minimum value, mass conservation and the energy of the numerical solutions produced by ETD2 schemes, from which the positivity preservation, mass conservation and energy stability are also well preserved, verifying the theoretical results of our numerical scheme.
\begin{figure}[!ht]
	\centerline{
		{\includegraphics[width=0.26\textwidth]{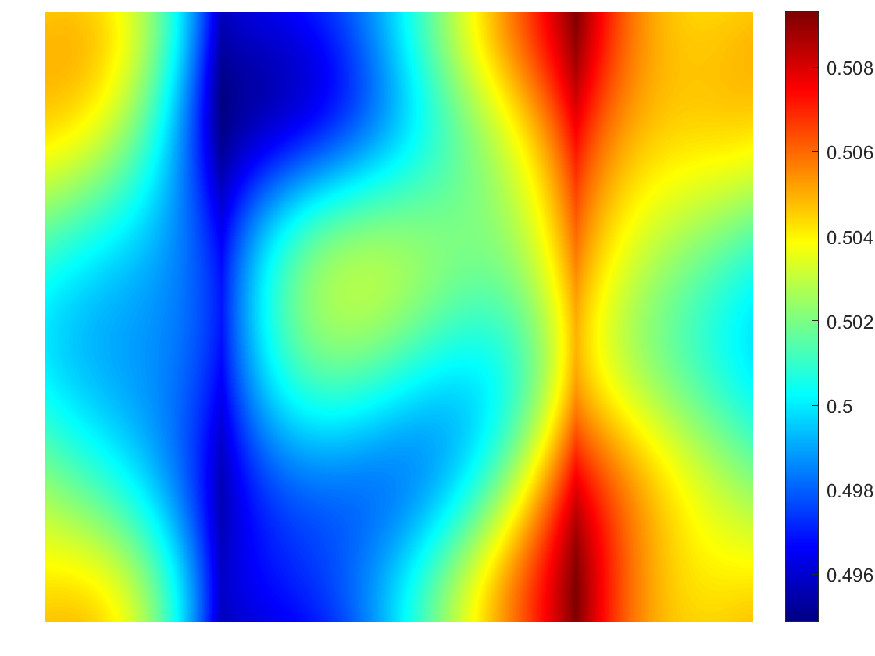}}\hspace{-0.1cm}
		{\includegraphics[width=0.26\textwidth]{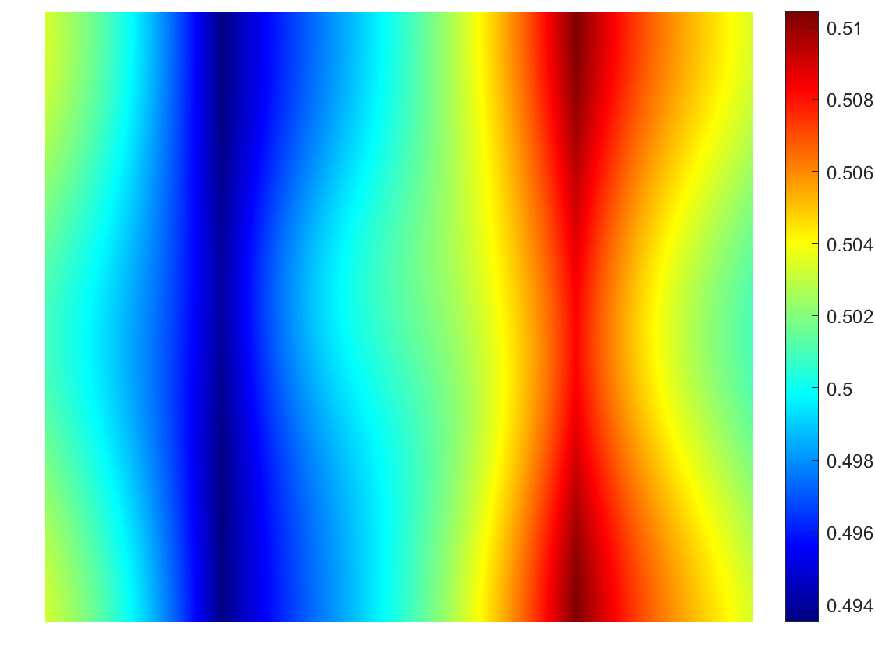}}\hspace{-0.1cm}		    {\includegraphics[width=0.26\textwidth]{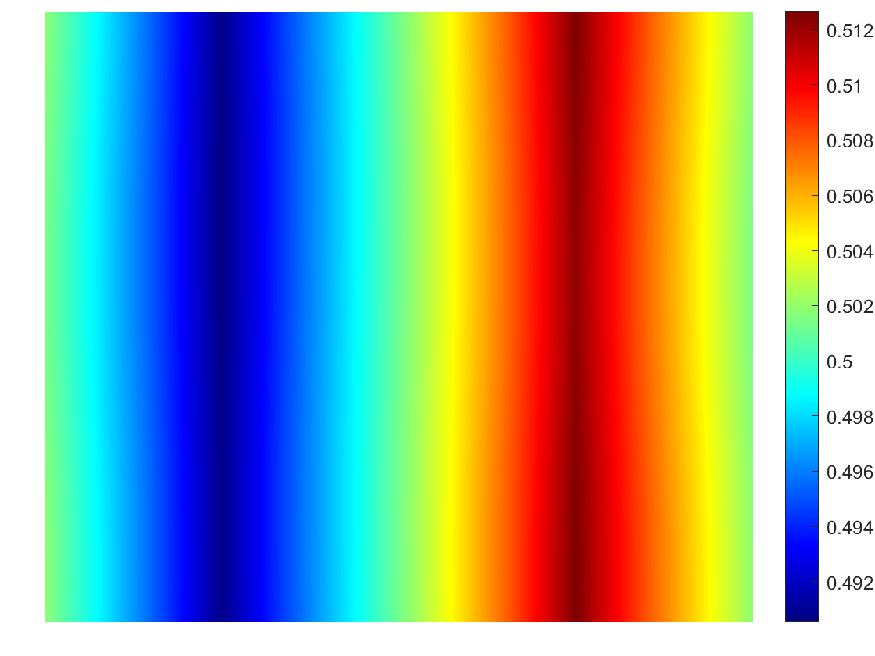}}\hspace{-0.1cm}
		{\includegraphics[width=0.26\textwidth]{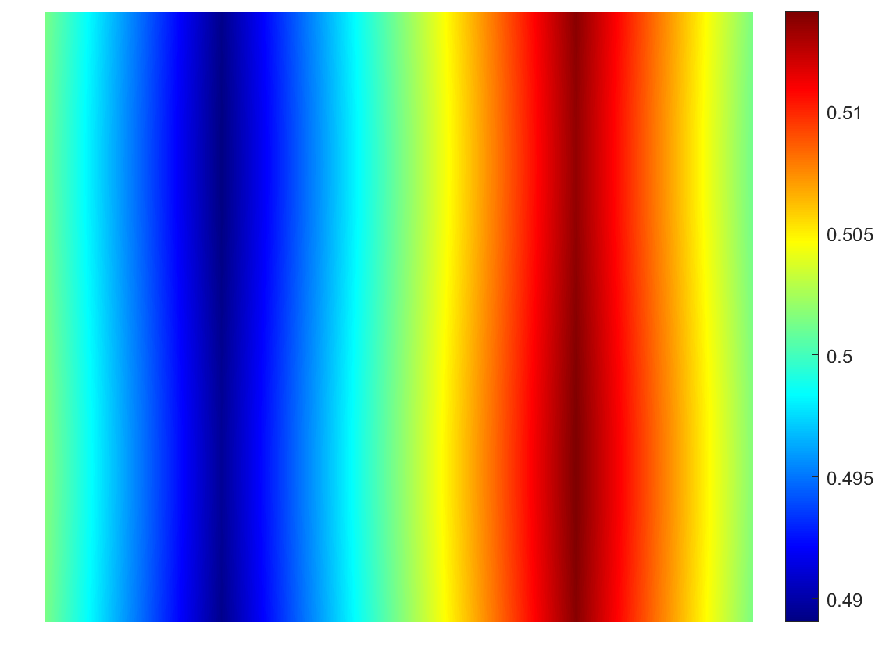}}}
		\centerline{
		{\includegraphics[width=0.26\textwidth]{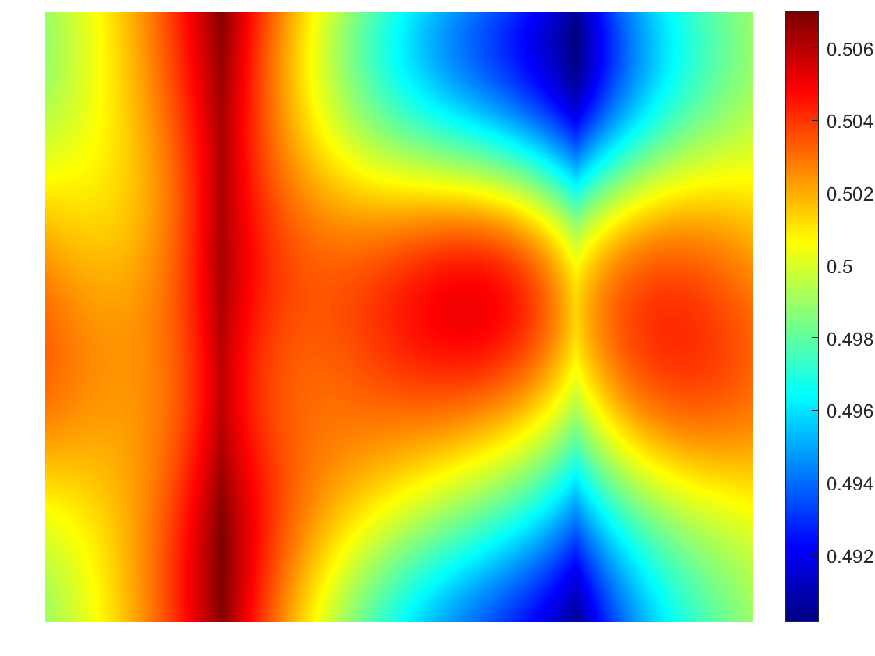}}\hspace{-0.1cm}
		{\includegraphics[width=0.26\textwidth]{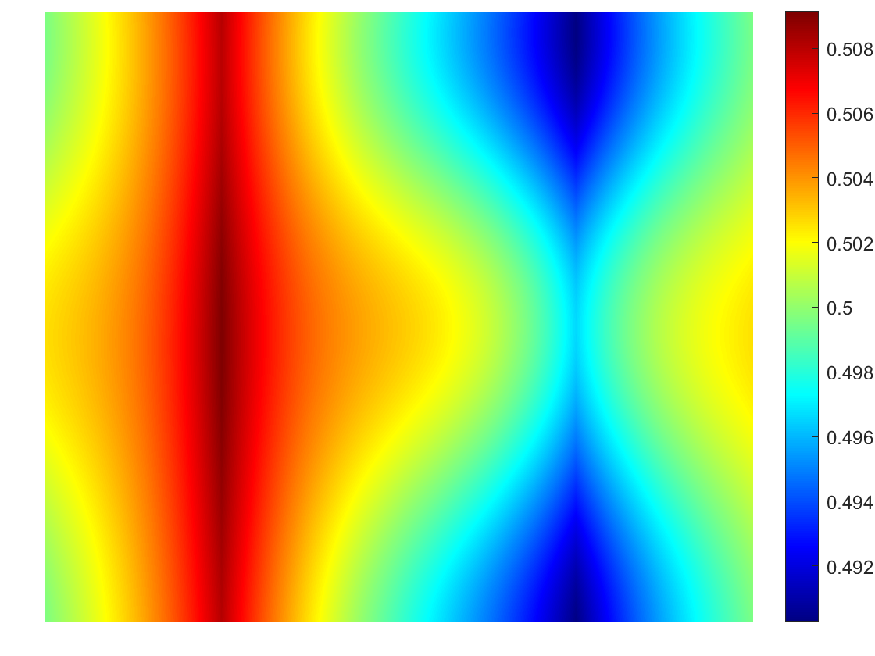}}\hspace{-0.1cm}		    {\includegraphics[width=0.26\textwidth]{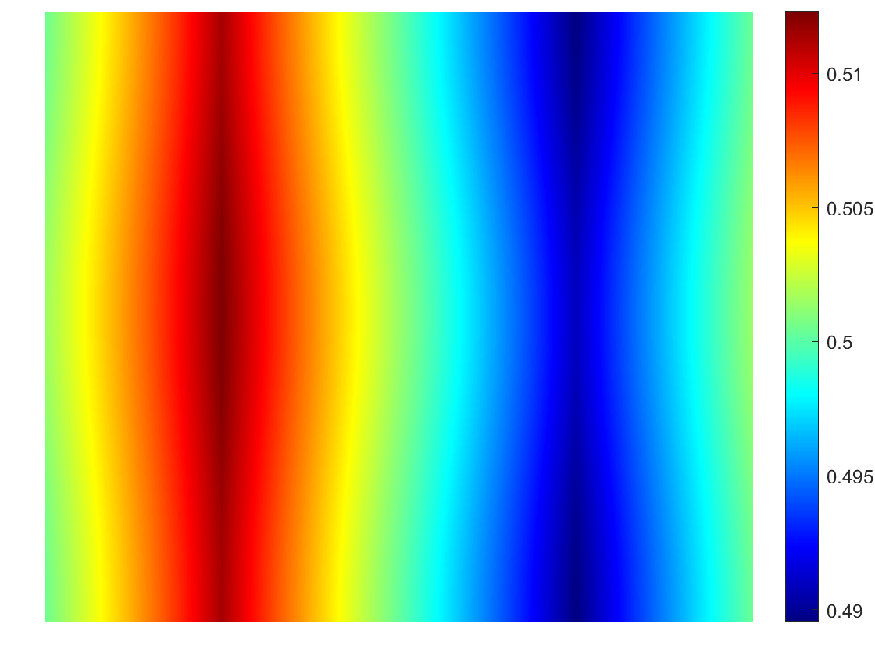}}\hspace{-0.1cm}
		{\includegraphics[width=0.26\textwidth]{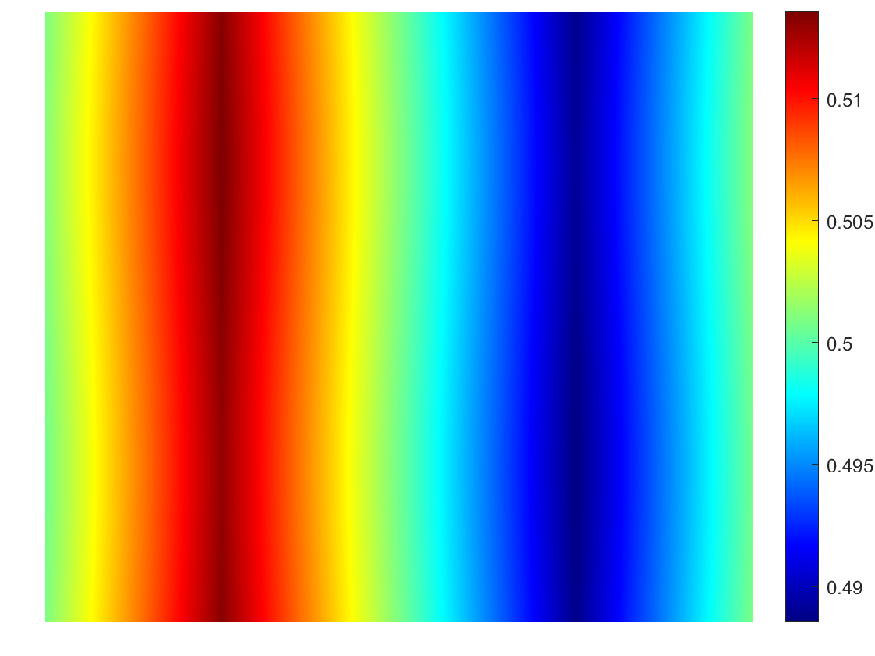}}}
	\caption{Profiles of concentrations produced by the ETD2 scheme with $\tau=0.01$ at $t$ = 0.01, 0.02, 0.05 and 0.3 (left to right) for the PNP equations with $\rho_0=50$. Top: positive icon $p$, bottom, negative icon $n$.}
	\label{saline_50_phase}
\end{figure}
\begin{figure}[!ht]
	\centerline{
		{\includegraphics[width=0.34\textwidth]{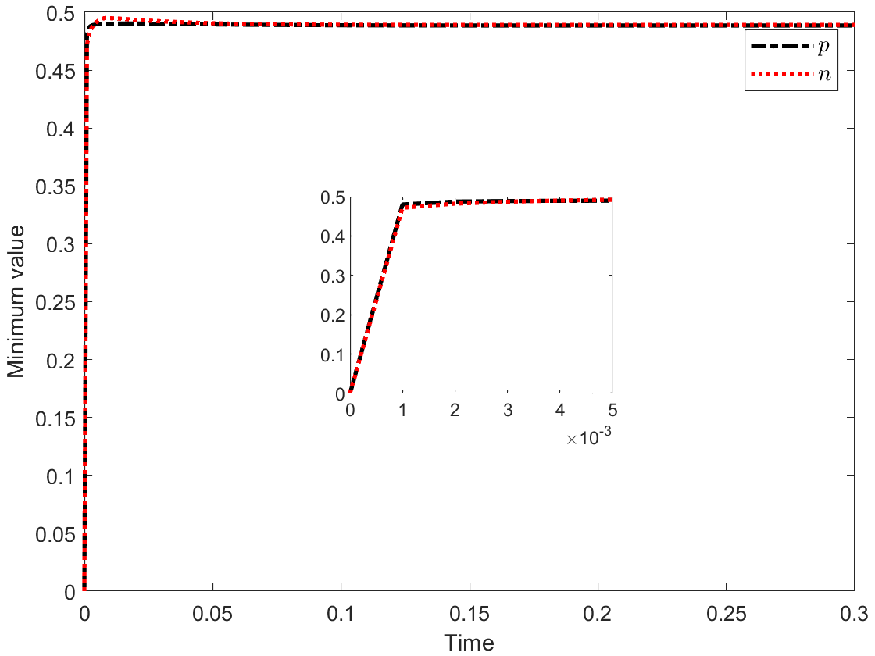}}\hspace{-0.1cm}
		{\includegraphics[width=0.34\textwidth]{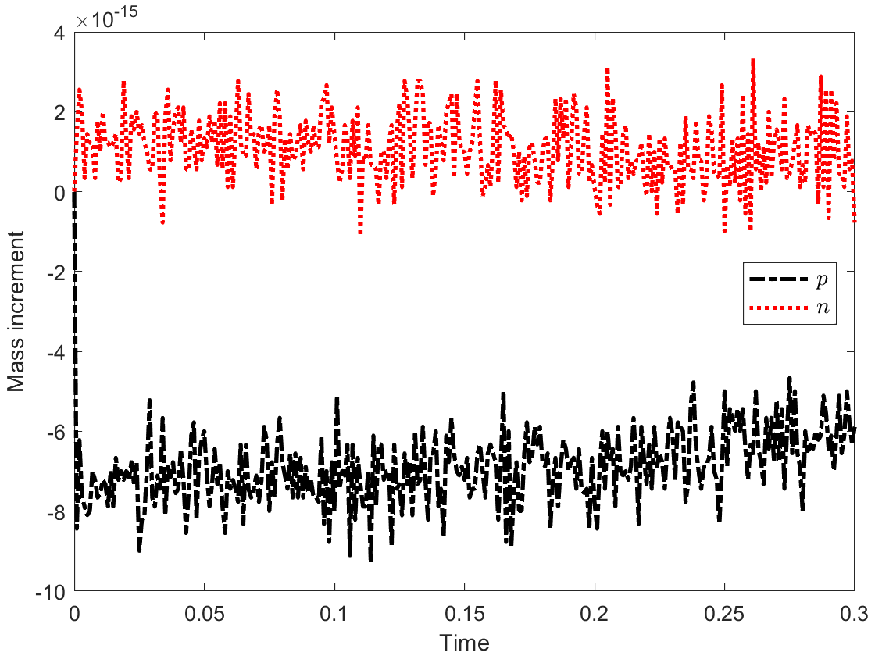}}\hspace{-0.1cm}
		{\includegraphics[width=0.34\textwidth]{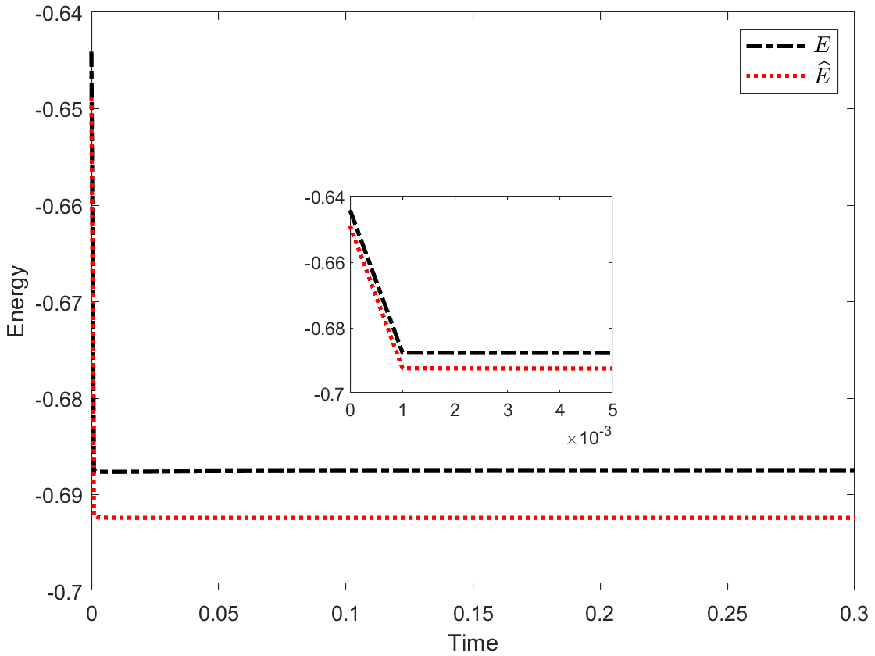}}}
	\caption{Evolutions of the minimum value, mass increment and the energy of numerical solution produced by the ETD2 scheme with $\tau=0.01$ for PNP equations with $\rho_0=50$.}
	\label{saline_50_phy}
\end{figure}
\section{Concluding remarks}\label{sec:con}
In this paper, we develop positivity-preserving and mass conservative linear schemes for the PNP equations, which are first- or second-order accurate in time. Based on the Slotboom transformation, we combine the convective and diffusion terms into a single self-adjoint elliptic operator, which allows for a quasi-symmetric spatial discretization by the second order finite difference method. Positivity preservation, mass conservation and energy stability are rigorously analyzed for the proposed schemes. A series of numerical examples are performed to confirm the theoretical findings and demonstrate the computational efficiency of the proposed schemes. 
Our ongoing work includes fully discrete error analysis for the ETD1 and ETD2 schemes as both the mesh size and the time step size approach zero, and application to more Wasserstein gradient flow model (e.g, Keller-Segel model) to design structure-preserving schemes.








\end{document}